\documentclass{article}
\usepackage{graphicx}
\usepackage{curves}
\usepackage{amsmath}
\usepackage{amssymb}
\usepackage{amsfonts}
\usepackage{amsthm}
\usepackage{longtable}
\usepackage{multirow}
\usepackage{color}

\newtheorem{lemma}{Lemma}[section]
\newtheorem{theorem}[lemma]{Theorem}
\newtheorem{prop}[lemma]{Proposition}

\newtheorem{coro}[lemma]{Corollary}

\newtheorem{remark}[lemma]{Remark}
\newtheorem*{teoA}{Theorem A}
\newtheorem*{teoB}{Theorem B}

\newcommand{\po}{\mathcal{P}}
\newcommand{\qo}{\mathcal{Q}}
\newcommand{\ko}{\mathcal{K}}
\newcommand{\G}{\Gamma}
\newcommand{\rank}{\mathrm{rank}}
\newcommand{\Aut}{\G}
\newcommand{\fl}{\mathcal{F}}
\newcommand{\Pyr}{\mathrm{Pyr}}
\newcommand{\Pri}{\mathrm{Pri}}
\newcommand{\W}{\mathcal{W}}
\newcommand{\Mon}{\mathcal{M}}

\numberwithin{equation}{section}

\begin{document}

\title{Products of abstract polytopes}

\author{Ian Gleason\\ {\small Univeristy of California
, Berkeley}\\{\small ianandreigf@gmail.com}
\\Isabel Hubard\\{\small Instituto de Matem\'aticas,} \\ {\small Universidad Nacional Autonoma de M\'exico}
\\ {\small isahubard@im.unam.mx}}
\date{}

\maketitle
\begin{abstract}
 Given two convex polytopes, the join, the cartesian product and the direct sum of them are well understood. In this paper we extend these three kinds of products to abstract polytopes and introduce a new product, called the topological product, which also arises in a natural way.
 We show that these products have unique prime factorization theorems. 
 We use this to compute the automorphism group of a product in terms of the automorphism groups of the factors and show that (non trivial) products are almost never regular or two-orbit polytopes.
We finish the paper by studying the monodromy group of a product, show that such a group is always an extension of a symmetric group, and give some examples in which this extension splits.
\end{abstract}

\section{Introduction}

In school we all dealt, in one way or another, with solids such as prisms and pyramids, but maybe also with bipyramids. The aim of this paper is to generalize these solids as different products of abstract polytopes, and study their symmetry and combinatorial properties.

Prisms, pyramids and bipyramids over polygons (see Figure~\ref{pentagons}) can be seen as a product of a polygon by either a segment or a point. However, these are three different kinds of products. While prisms are the cartesian product of a segment with a polygon, pyramids are the join product of a point with a polygon and bipyramids are the direct product of a segment with a polygon.
In the theory of convex polytopes the generalization of these three notions are the cartesian product, the join product and the direct sum, respectively (\cite{convex}).
Given two convex polytopes $\po\subset\mathbb{R}^n$ and $\qo\subset\mathbb{R}^m$, their products are defined as follows.
\begin{figure}[htbp]
\begin{center}
\includegraphics[width=10cm]{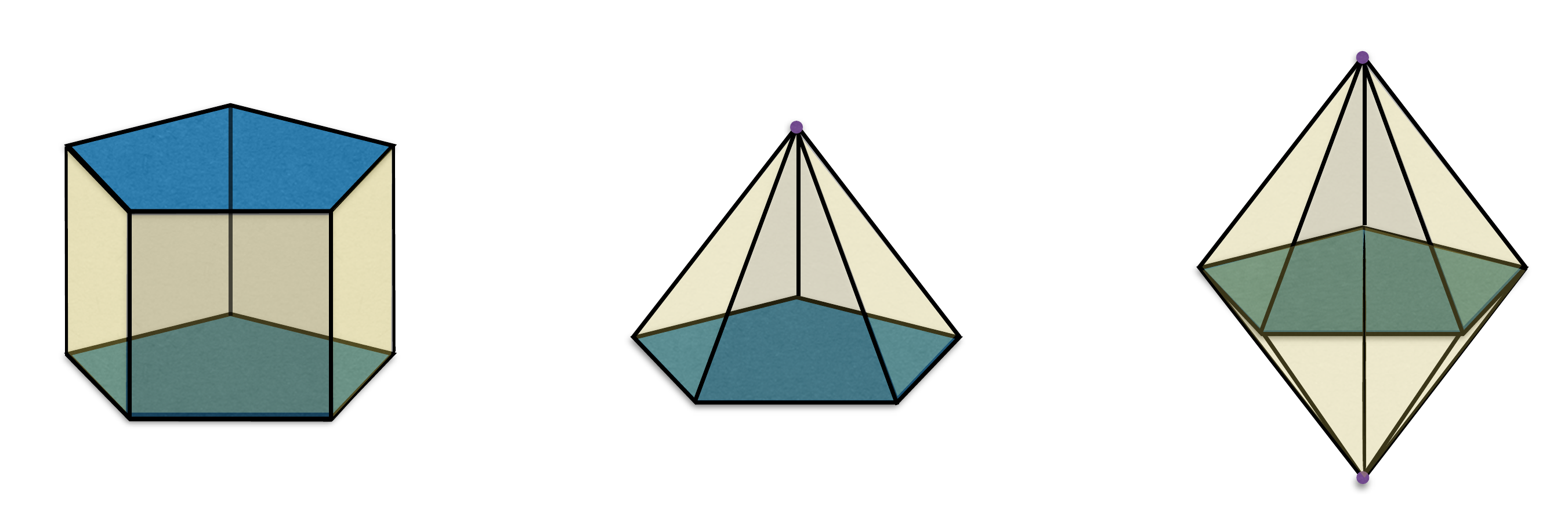} 
\caption{A prism, pyramid and bipyramid over a pentagon.}
\label{pentagons}
\end{center}
\end{figure}

The join of $\po$ and $\qo$  is obtained by embedding $\po$ and $\qo$ in disjoint affine subspaces of $\mathbb{R}^{n+m+1}$ and taking the convex hull of their vertices. For example, for each $d\geq 1$, a $d$-simplex can be seen as the join of a point and a $(d-1)$-simplex (Figure~\ref{JoinPowersofaPoint}).
\begin{figure}[htbp]
\begin{center}
\includegraphics[width=10cm]{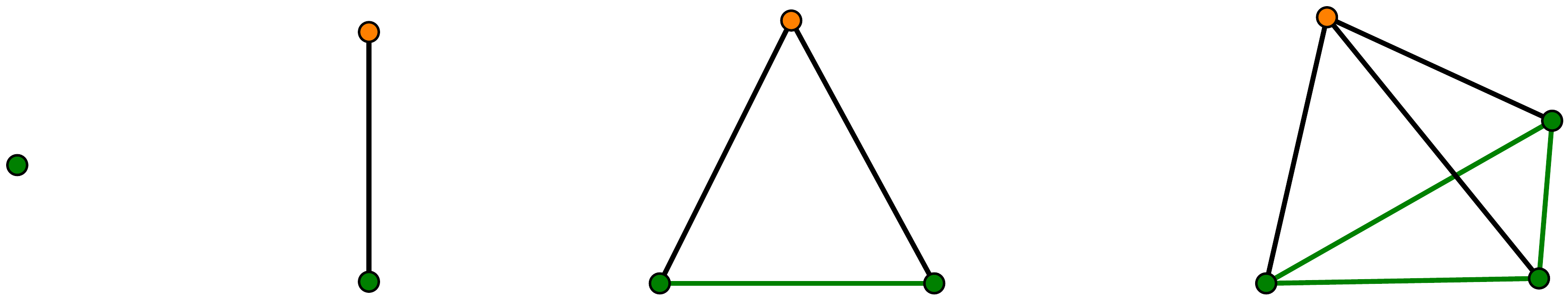}
\caption{A $d$-simplex is the join product of a point and a $(d-1)$-simplex.}
\label{JoinPowersofaPoint}
\end{center}
\end{figure}

 The cartesian product of $\po$ and $\qo$ is obtained by taking the convex hull of $V(\po) \times V(\qo)$ in $\mathbb{R}^{n+m}$. The classical example in this case, is to see a $d$-cube as the cartesian product of an edge -or line segment- with a $(d-1)$-cube (as in Figure~\ref{CartesianPowersofaEdge}).
 \begin{figure}[htbp]
\begin{center}
\includegraphics[width=10cm]{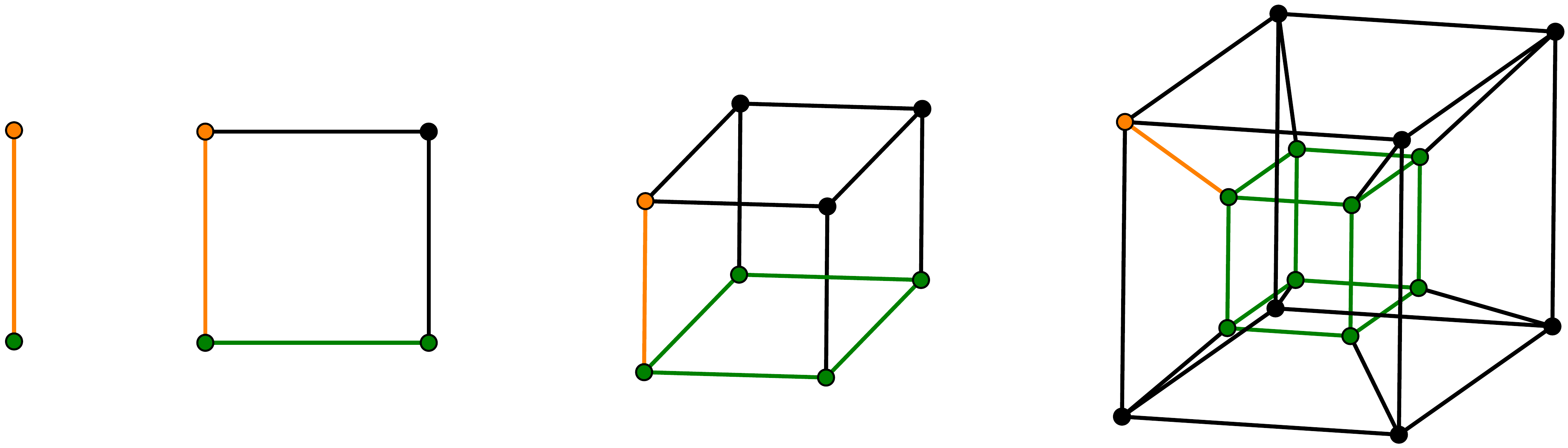}
\caption{A $d$-cube is the cartesian product of an edge and a $(d-1)$-cube.}
\label{CartesianPowersofaEdge}
\end{center}
\end{figure}

 The direct sum of  $\po$ and $\qo$ is slightly more complicated to state. We first  require that $\po$ and $\qo$ contain in their relative interiors the origins of $\mathbb{R}^n$ and $\mathbb{R}^m$, respectively. Then the direct sum is the convex hull of all the points of the form $(v,0)$ and $(0,u)$, where $v\in V(\po)$ and $u\in V(\qo)$. For example, cross polytopes can be generated in this way.

Note that as where in the join product and cartesian product of convex polytopes $\po$ and $\qo$, every face of $\po$ and of $\qo$ is again a face of the product, for the direct sum this is no longer the case. On the other hand, for both the join product and the direct sum, the vertices of the product is the union of the vertices of both polytopes, while the cartesian product of two polytopes, in general, has more vertices.
 It is straightforward to see that the only convex polyhedra (or convex 3-polytopes) that arise as one of these products are precisely the prisms, the pyramids and the bipyramids over polygons.

It is also well-know that, in $\mathbb{R}^4$, the product of two orthogonal circles $\mathbb{S}^1 \times \mathbb{S}^1$ is precisely the {\em flat torus} (also known as the Clifford torus, \cite{clifford}). If we place $n$ points on each of the circles, evenly spaced, we obtain $n$ congruent line segments on each circle. Then,
 take the cartesian product of each point of each $\mathbb{S}^1$ with line segment of the other $\mathbb{S}^1$. What you obtain is a tessellation of the flat torus by squares (see Figure~\ref{torus}). Hence, some maps on surfaces can also be seen as products of polygons.
 
 \begin{figure}[htbp]
\begin{center}
\includegraphics[width=6cm]{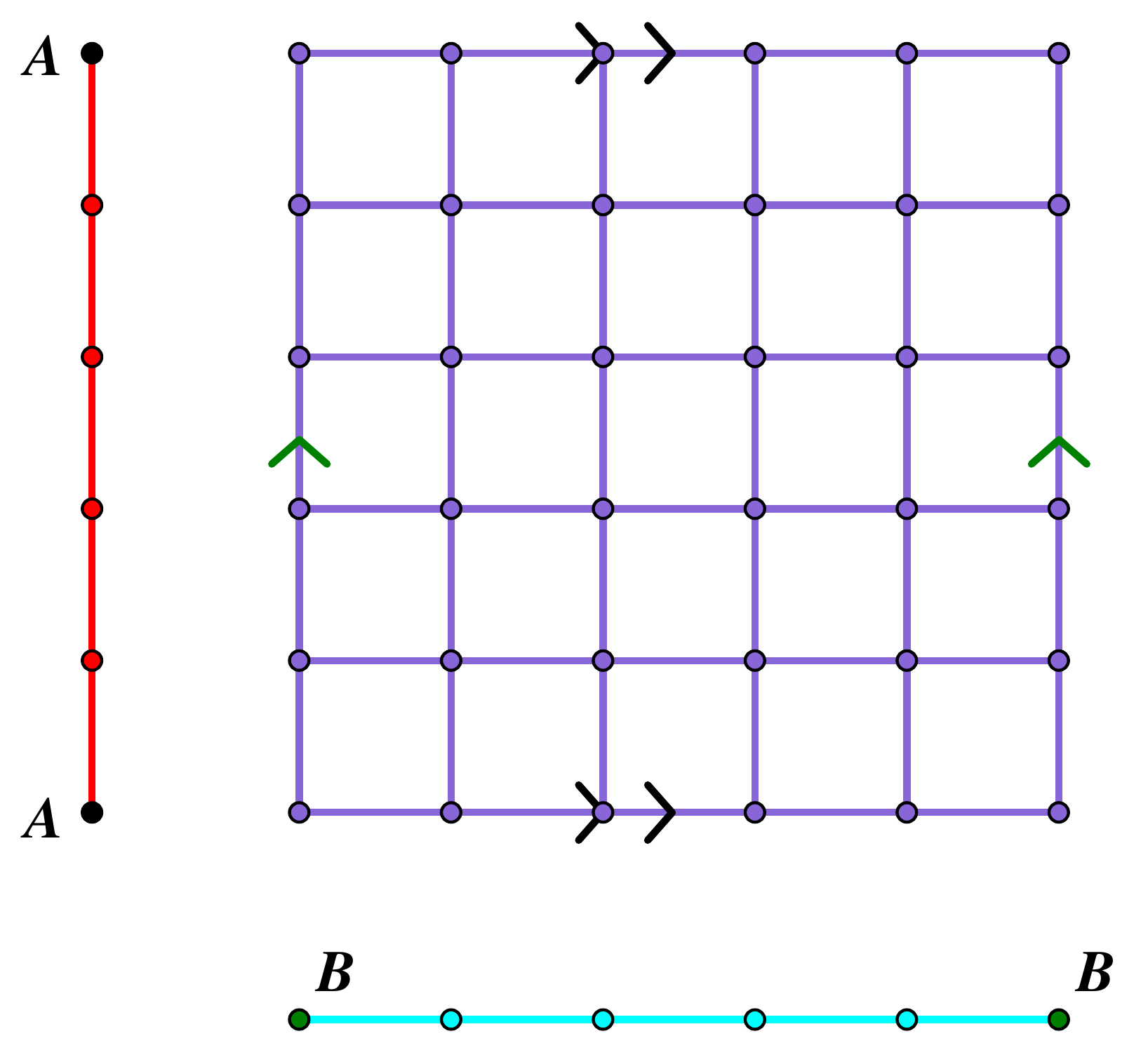}
\caption{The cartesian product of two pentagons can be seen as a tessellation of the torus by squares.}
\label{torus}
\end{center}
\end{figure}

 Abstract polytopes generalize the (face lattice) of convex polytopes. Moreover, they also generalize non-degenerated maps. Hence, it is natural to generalise the four products described above and define them for abstract polytopes, and we do so in Section~\ref{sec:allproducts}. As we show, the four products are closed for abstract polytopes, meaning that  the product of two abstract polytopes is again an abstract polytope (under any of the four products). We shall also study, for each product, which polytopes are trivial, in the sense that the product of them with any  polytope $\po$ is simply $\po$. With that in mind, it is natural to say that a polytope is prime with respect to a given product, if it cannot be decomposed as the product of non-trivial polytopes. 
 We show a unique prime factorization theorem and use it to investigate the structure of the automorphism group of a product. Theorem A summarizes the main results of Sections~\ref{sec:allproducts}, \ref{sec:fact} and \ref{sec:auto}.
 
 \begin{teoA}
Let $\po$ and $\qo$ be two abstract polytopes and $\odot$ be a product of polytopes (either the direct sum, the join, cartesian or topological product). Then,
\begin{enumerate}
\item[a)] The product $\po \odot \qo = \qo \odot \po$ is an abstract polytope. In particular, $\po \odot \po \odot \dots \odot \po=: \po^m$ is also an abstract polytope.
\item[b)]  The polytope $\po$ can be uniquely factorized as a product of prime polytopes.
\item[c)]  If $\po =  \qo_1^{m_1} \odot \qo_2^{m_2}\odot \dots \odot \qo_r^{m_r}$, where the $\qo_i$ are distinct prime polytope with respect to $\odot$, then
$$ \Aut(\po) = \Pi_{i=1}^r (\Aut(\qo_i)^{m_i}\rtimes S_{m_i}).$$
\item[d)] If $\po$ is decomposed as above, where $\qo_i$ is a $k_i$-orbit polytope, then $\po$ is a $k$-orbit polytope, with
$$ k = \frac{(\sum_{i=1}^{r} m_i n_i) !  \Pi_{i=1}^{r} k_i^{m_i} \frac{(m_i n_i)!}{(n_i)!^{m_i} m_i !}}{\Pi_{i=1}^r (m_in_i)!},$$ where if $\odot$ is the join product, then $n_i$ is the rank of $\qo_i$, if $\odot$ is either the cartesian product or the direct sum, then $n_i+1$ is the rank of $\qo_i$, and if $\odot$ is the topological product, then $n_i+2$ is the rank of $\qo_i$
\end{enumerate}
\end{teoA}
 As a corollary of part $d)$ of Theorem A, for each product, we also obtained the families of regular and two-orbit  polytopes that are not prime.  
 
The monodromy group of a polytope (either convex or abstract) encodes all the combinatorial information of the polytope. 
It was first studied by Hartley in \cite{hartely99} and he used it to construct regular covers of (non-regular) polytopes. 
It is well-know that the monodromy group of a regular polytope is isomorphic to its automorphism group. However, little is known about monodromy groups of non-regular polytopes.
In the last decade, there has been an effort to understand these groups (see for example \cite{berman2014monodromy}, \cite{hartley2012minimal}, \cite{mixmono}).
In particular in \cite{hartley2012minimal} Hartely et al. study the monodromy group of the prism oven an $n$-gon and compute it, in terms of generators and relations. Moreover, in \cite{berman2014monodromy}, Berman et al. study that of the pyramid over an $n$-gon and show that it is an extension of the symmetric group $S_4$ by a cyclic group which sometimes splits (and determine when).
We show that products of polytopes are useful to understand monodromy groups of some non-regular polytopes. 
 
The results of monodromy groups of polytopes are summarized in the following theorem

\begin{teoB}
Let $\qo_1, \dots \qo_r$ be polytopes of ranks $n_1, n_2, \dots, n_r$, respectively and let $\odot$ be a product of polytopes. Let $\po=\qo_1\odot \qo_2 \odot \dots \odot \qo_r$. 
Then,
\begin{enumerate}
\item[a)] The monodromy group $\Mon(\po)$ is an extension of $S_n$, where $n=\Sigma_{i=1}^r n_i +c$, and $c=r$, if $\odot=\Join$, $c=0$ if $\odot=\times, \oplus$ and $c=-r$ if $\odot=\square$.
\item[b)] The extension of $a)$ splits (at least) in the following cases: 
\begin{itemize}
\item If $\po$ is the prism (or the bipyramid) over an $n$-gon. In this case $\Mon(\po) \cong K \rtimes S_3$, where $K$ is an extension $(C_2)^3$ by $(C_m)^3$ ($m=\frac{n}{gcd(n,4)}$); moreover, this extension splits whenever $n$ in not congruent to $0$ modulo $8$, in which case $\Mon(\po) \cong \big( (C_2)^3 \rtimes (C_m)^3\big) \rtimes S_3$.
\item If $\po$ is the prism (or the bipyramid) over a $3$-polytope having the property that all its vertex figures (faces) are isomorphic to an $n$-gon, with $n$ not congruent to $0$ modulo $9$;
\item If $\odot=\square$ and each $\qo_i$ has rank 2.
\end{itemize}
\end{enumerate}
\end{teoB}
By \cite{berman2014monodromy}, we already know that the monodromy group of a pyramid over an $n$-gon is an extension of  $S_4$ by $(C_m)^4$, where $m=\frac{p}{gcd(3,p)}$, and that such extension splits whenever $n$ is not congruent to $0$ modulo $9$. Our techniques to show Theorem B can also be used to show this result.

\section{Abstract polytopes}\label{sec:abpol}

Abstract polytopes generalise the (face-lattice) of classical polytopes as combinatorial structures. In this chapter we review the basic theory of abstract polytopes and refer the reader to \cite{arp} for a detail exposition of the subject.

An {\em (abstract) $n$-polytope} (or {\em (abstract) polytope of rank $n$}) $\po$
is a partially ordered set whose elements 
are called {\em faces} 
that satisfies the following properties. 
It contains a minimum face $F_{-1}$ and a maximum face $F_n$.  
These two faces are the {\em improper} faces of $\po$; all other faces are said to be {\em proper}.
There is a rank function from $\po$ to the set $\{-1, 0, \dots, n\}$ such that $\rank(F_{-1})= -1$ and
$\rank(F_n)=n$. 
The faces of rank $i$ are called {\em $i$-faces}, the $0$-faces
are called {\em vertices}, the $1$-faces are called {\em edges} and the
$(d-1)$-faces are called {\em facets}.
Every maximal totally ordered subset (called {\em flag}) contains precisely $n+2$ elements including $F_{-1}$ and $F_n$. 
If $\Phi$ is a flag of $\po$ we shall often denote by $\Phi_i$ the $i$-face of $\Phi$.
For incident faces $F \le G$, we define the {\em section} $G/F := \{H \,|\,F\le H\le G\}$, and when convenient, we identify the section $F/F_{-1}$ with the face $F$ itself in $\mathcal P$. 
The section $F_n/F_0 := \{H \,|\, H\ge F_0\}$, when $F_0$ is a vertex, is called the {\em vertex-figure} of $\mathcal P$ at $F_0$, and if $F_i$ is a face of rank $i>0$, then $F_n/F_i$ is a {\em co-face} of $\po$.
All sections $G/F$ of $\po$ are by themselves posets with a rank function, minimum and maximum faces and satisfy that all their maximal chains have the same number of elements.
A section $G/F$ is said to be {\em connected}, if $\rank(G)-\rank(F)\leq 2$ or if whenever $F', G' \in G/F$, with $F', G' \neq F, G$, there exists a sequence of faces
$$F'=F^0,F^1,F^2, \dots F^k=G',$$
such that $F< F^i <G$  and either $F^i\leq F^{i+1}$ or $F^{i+1}\leq F^{i}$, for every $i=0, \dots k$.
We further ask that $\po$ is {\em strong connectivity}, meaning that all the sections of $\po$, including itself, are connected.
The last condition that $\po$ should satisfy to be a polytope, known as the {\em diamond condition}, is the following.
If $F$ and $G$ are incident faces such that $\rank(G) - \rank(F) = 2$, then there exist precisely two faces $H_1$ and $H_2$ such that $F < H_1, H_2 < G$. 
This property implies that
for any flag $\Phi$ and any $i \in \{0, \dots, n-1\}$ there exists a unique flag
$\Phi^i$ differing from $\Phi$ only in the $i$-face.
The flag $\Phi^i$ is called the {\em $i$-adjacent flag} of $\Phi$. 

Up to isomorphism, there is a unique $n$-polytope for $n=0,1$. Polygons (including infinite ones) are $2$-polytopes and no-degenerated maps are 3-polytopes. In general, a convex $d$-polytope can be regarded as an abstract $d$-polytope.

It is not difficult to see that the diamond condition implies that the incidence structure consisting of all vertices and edges of a polytope $\po$, together with the incidence given in $\po$ is a graph. We shall refer to this graph as the {\em $1$-skeleton} of $\po$.


 An {\em automorphism} of a polytope is an order-preserving permutation of its faces. 
 We denote the group of automorphisms of $\po$ by $\G({\po})$. 
It is straightforward to see that $\G(\po)$ acts on the set of flags, denoted by $\fl(\po)$, in the natural way. Moreover,  the strong connectivity of $\po$ implies that such action is free (or semi-regular).

An $n$-polytope $\mathcal P$ is said to be {\em regular} whenever $\G({\po})$ acts transitively on the flags. 
We say that $\po$ is  a {\em $k$-orbit polytope} if $\G({\po})$ has precisely $k$ orbits on $\fl(\po)$. (Hence, regular polytopes and 1-orbit polytopes are the same.)

Given a polytope $\po$, one can define the dual of $\po$, denoted by $\po^*$, as the poset whose elements coincide with the elements of $\po$, but the order is reversed. In other words, $\po^*$ is the dual of $\po$ if there exists a bijection $\delta: \po \to \po^*$ that reverses the order. Note that $(\po^*)^* \cong \po$ and that $\G(\po) \cong \G(\po^*)$.


The {\em monodromy group} $\Mon(\po)=\langle r_0, r_1, \dots r_{n-1}\rangle$ of an $n$-polytope $\po$ is the subgroup of the permutations of the set of flags $\fl(\po)$ that is generated by the permutations $r_i:\Phi \mapsto \Phi^i$ (see \cite{hartely99, hubard2009monodromy}). 
The elements of the monodromy group are, in general, far from being automorphisms of the polytope. 
By the connectivity of $\po$, $\Mon(\po)$ is transitive on $\fl(\po)$.
One can think of the generators of the monodromy group as the instructions to assemble the flags of the polytope. 
In fact the monodromy group possesses all the combinatorial information of the polytope.
Given $w\in\Mon(\po)$, $\gamma\in\G(\po)$ and $\Phi \in \fl(\po)$ it is straightforward to see that $(\Phi w) \gamma = (\Phi\gamma)w$.

The generators $r_0, r_1,\dots,r_{n-1}$ of $\Mon(\po)$ are involutions and satisfy, at least, the relations $r_ir_j=r_jr_i$ whenever $|i-j|>1$. 
Whenever $\po$ is a regular polytope, its monodromy group and its automorphism group are isomorphic. However, in other cases little is known about the structure of the monodromy group of a polytope (see \cite{mixmono} for further discussion on the subject).

\subsection{Hasse diagram}
Given a poset $\po$ and $F,G \in \po$, we shall say that {\em $F$ is covered by $G$} if $F<G$ and there exists no $H \in \po$ such that   $F< H < G$. 
In particular, if $\po$ is a polytope, then a face $F$ is covered by a face $G$ whenever $F<G$ and $\rank (G) -\rank (F) = 1$.
The {\em Hasse diagram} of the poset $\po$, denoted by $H(\po)$ is the directed graph whose vertices are the elements of $\po$ and there is an arc from a face $G$ to a face $F$ whenever $F$ is covered by $G$.

Note that if $\po$ is a polytope, then the digraph $H(\po)$ has one sink, one source, is acyclic and  directed paths of maximal length have $n+2$ vertices. Moreover, $F < G$ in $\po$ if and only if there is a directed path from $G$ to $F$ in $H(\po)$. Therefore if $\po$ and $\qo$ are two polytopes such that there exists and isomorphism between $H(\po)$ and $H(\qo)$, then it induces an isomorphism between $\po$ and $\qo$ (and viceversa: isomorphisms between $\po$ and $\qo$ induce isomorphisms between their Hasse diagrams). Note further that $\G(\po) \cong \Aut(H(\po))$.

A poset $\po$ is said to be {\em discrete} if the transitive closure of the Hasse diagram $H(\po)$ is $\po$ itself. For example $\mathbb{Z}$ is a discrete poset, while $\mathbb{Q}$ is not. All abstract polytopes are discrete posets. Hence, in this paper, unless otherwise indicated, all posets are discrete.

\section{Product of posets and digraphs}\label{sec:productsgraphs}

As we have seen before, one can identify an abstract polytope with its Hasse diagram. For the purpose of this paper it shall prove helpful to often think of abstract polytopes as directed graphs (with the induced properties). Thus, we study here some properties about products of posets and digraphs.

Given posets $\qo_i$ with $i \in I$, the {\em (cardinal) product}, $\Pi_{i \in I}\qo_i$ is the ordered set on their Cartesian product, with component-wise order. 
Denoting by $\po*\qo$ the product of two posets $\po$ and $\qo$, it is then straightforward that $\po*\qo=\qo*\po$ and that, 
if $\mathcal K$ is yet another poset, then $\po*(\qo*{\mathcal K})=(\po*\qo)*{\mathcal K}=:\po*\qo*{\mathcal K}$. We denote by $\po^k$ to the product of $k$ copies of $\po$.

Given $F, G \in \po$, with $F \leq G$, the set $\{H \in \po \mid F\leq H \leq G\}$ is the {\em closed interval} between $F$ and $G$. Similarly, $\{H \in \po \mid F< H < G\}$ is said to be an {\em open interval}. (Hence, sections of a polytope are in fact closed intervals of the poset.)
 If the poset $\po$ does not have a minimum or a maximum, the sets $\{H \in \po \mid F\leq H\}$, 
$\{H \in \po \mid F\geq H\}$ , $\{H \in \po \mid F< H\}$ and $\{H \in \po \mid F> H\}$ are also said to be (closed/open, resp.) intervals of $\po$.
We say that a poset $\po$ is {\em factorable} if there exist non-trivial posets $\po_1$ and $\po_2$ such that $\po=\po_1*\po_2$ and that $\po$ is {\em prime} if no such factorization exists.

 In \cite{Hash2} Hashimoto shows that if we have  two proper factorizations of a poset $\po$, then there exists another proper factorization of $\po$ that is a refinement of the two original ones. 
 Hashimoto then uses this result to show the following theorem. 
 
\begin{theorem}[\cite{Hash2}]
\label{uniquefactposet}
Every connected poset has a unique prime factorisation (up to isomorphism).
\end{theorem}

In this context, a poset $\po$ is said to be {\em connected} if for any two elements $F,G \in \po$, there exists a sequence $F=F_0, F_1, \dots F_k=G$ such that $F_i \leq F_{i-1}$ or $F_i \geq F_{i-1}$ for every $i\in \{1, \dots, k\}$.

Recall now that a graph $G$ is said to be connected if for any two vertices $u$ and $v$, there is a $u$-$v$ path in $G$. A digraph is said to be {\em weakly connected} if its underlying graph (where the arcs are replaced by edges) is connected. 
 A discrete poset is connected if and only if is its Hasse diagram is weakly connected.

Given two digraphs $\mathcal{D}_1= (V_1, E_1)$ and $\mathcal{D}_2= (V_2, E_2)$, the {\em cartesian product} of $\mathcal{D}_1$ and $\mathcal{D}_2$ is a digraph $\mathcal{D} = \mathcal{D}_1 \times \mathcal{D}_2$ whose vertex set is $V(\mathcal{D})=V_1 \times V_2$, and such that there is an arc $(v_1,v_2) \to (w_1,w_2)$ if $v_1=w_1$ and $v_2 \to w_2 \in A_2$ or if $v_2=w_2$ and $v_1 \to w_1 \in A_1$. 

It was shown in \cite{rival} that the Hasse diagram of a product of orders is a product of Hasse diagrams. In fact we have the following proposition. 

\begin{prop}[\cite{rival}]
Let $\qo_i$, with $i \in I$, be a family of posets. If $\po \cong \Pi \qo_i$, then $H(\po) \cong \Pi H(\qo_i)$.
\end{prop}

Moreover, in \cite{walker}, Walker showed that if $\po$ is a poset such that $H(\po) = \Pi_{i \in I} G_i$, for some digraphs $G_i$, then there exist posets $\qo_i$,  such that $H(\qo_i) = G_i$ for each $i \in I$ and $\Pi_{i \in I}\qo_i = \po$.

Given a digraph $\mathcal D$, if there exists digraphs $\mathcal{D}_1$ and $\mathcal{D}_2$ such that $\mathcal{D}=\mathcal{D}_1 \times \mathcal{D}_2$, where $|V_1|,|V_2|>1$, then we say that $\mathcal{D}$ is {\em cartesian-factorable} (or, simply, {\em factorable}) and that $\mathcal{D}=\mathcal{D}_1 \times \mathcal{D}_2$ is a {\em proper factorization} of $\mathcal{D}$. If no such factorization exists, we shall say that $\mathcal{D}$ is {\em prime}. 

Hence, a poset $\po$ is prime if and only if its Hasse diagram $H(\po)$ is a prime digraph.

\section{Products of polytopes}\label{sec:allproducts}

As pointed out in the introduction, geometrically, there are several kinds of products of polytopes. In this section we define each of them as products of abstract polytopes. We shall see that although the different products that we define have different geometric interpretations, they can all be expressed in terms of cardinal products of posets.

\subsection{Join product}

Geometrically the most natural product might be the cartesian one, however when considering abstract polytopes the natural product arrises from the product of posets. We therefore start by studying such product of polytopes.

Given two polytopes $\po$ and $\qo$, the {\em join product} of $\po$ and $\qo$, denoted $\po \Join \qo$, is defined as the set
\begin{eqnarray}
\label{setjoin}
 \po \Join \qo = \{(F,G)  \mid  F \in \po , G \in \qo\},
 \end{eqnarray}
 where the order is given by 
 \begin{eqnarray}
 \label{orderjoin}
 (F, G) \leq_{\po \Join \qo} (F', G') \ \mathrm{ if} \ \mathrm{ and} \  \mathrm{ only} \ \mathrm{ if}  \ F \leq_{\po} F' \ \mathrm{ and } \ G \leq_{\qo} G'.
  \end{eqnarray}
  
  In other words, $\po \Join \qo$ is simply the cardinal product $\po*\qo$ of the posets $\po$ and $\qo$.
(We have changed the notation as we shall only use the join product $\po \Join \qo$ when both $\po$ and $\qo$ are polytopes, while we shall keep referring to the product $\po * \qo$ as the product of any two posets.)
It is therefore straightforward to see that $\po \Join \qo$ is indeed a poset. 
Moreover,
  $\po \Join \qo = \qo \Join \po$ and  if $\mathcal K$ is another abstract polytope, then $(\po \Join \qo) \Join {\mathcal K} = \po \Join (\qo \Join {\mathcal K}) = \po \Join \qo \Join {\mathcal K} $. Hence, for every natural number $k$, $\po^k$ denoted the join product of $\po$, $k$-times.
Observe further that a section of $\po \Join \qo$ is the join of a section of $\po$ and a section of $\qo$. 
That is,

\begin{lemma}
Let $\po$ and $\qo$ be two polytopes and consider the join $\po \Join \qo$. Let $f,F\in \po$, $g,G \in \qo$ such that $f\leq F$ and $g \leq G$. Then 
$$ (F,G)/(f,g) \cong F/f \Join G/g.$$
\end{lemma} 

If  $\po$ and $\qo$ are two polytopes of ranks $n$ and $m$ respectively,
then the rank functions of $\po$ and $\qo$ naturally induce a rank function on $\po \Join \qo$, namely, $$\rank_{\po \Join \qo} (F,G) = \rank_{\po}(F) + \rank_{\qo} (G) +1.$$ 
Hence, the rank function of $\po \Join \qo$ has range from $-1$ to $n+m+1$, and therefore $\po\Join\qo$ shall have rank $n+m+1$.

It is not difficult to see that, if $P_{-1}$ and $Q_{-1}$ denote the minimal faces of $\po$ and $\qo$ respectively, then the vertices of $\po \Join \qo$ are of the form $(P_{-1}, v)$ or $(u, Q_{-1})$, where $v$ is a vertex of $\qo$ and $u$ is a vertex of $\po$. Hence, the vertices of $\po \Join \qo$ are in bijection with the union of the vertices of $\po$ and $\qo$. In general, for $1 \leq i \leq n-1$, if $F$ is an $i$-face of either $\po$ (or $\qo$), then $(F, Q_{-1})$ (or $(P_{-1},F)$) is an $i$-face of $\po \Join \qo$, but these are not all the $i$-faces of $\po \Join \qo$.

  \begin{prop}
Let $\po$ and $\qo$ be two polytopes of ranks $n$ and $m$, respectively. Then $\po\Join\qo$ is a polytope of rank $n+m+1$.
\end{prop}

\begin{proof}
First note that, if $F_{-1}$ and $G_{-1}$, and $F_{n}$ and $G_m$ are the minimal and maximal faces of $\po$ and $\qo$ respectively, then $(F_{-1}, G_{-1})$ is the minimal face of $\po \Join \qo$, while $(F_n, G_m)$ is its maximal face.

Given two elements $(F_i, G_a), (F_j, G_b) \in \po \Join \qo$ such that $(F_i, G_a) \leq_{\po \Join \qo} (F_j, G_b)$, 
there are flags $\Phi$ and $\Psi$ of $\po$ and $\qo$, respectively, such that $F_i, F_j \in \Phi$, and $G_a, G_b \in \Psi$.
It is straightforward to see
that the set
$$\{(F_i, G_a)=(\Phi_i, \Psi_a), (\Phi_{i+1}, \Psi_a), \dots, (\Phi_j, \Psi_a), (\Phi_j, \Psi_{a+1}), \dots (\Phi_j, \Psi_b)=(F_j, G_b)\}$$ (which is a subset of $\po\Join\qo$),
is a chain of the order $\po\Join\qo$ that has one element of each rank from $\rank_{\po\Join\qo}(F_i, G_a)$ to $\rank_{\po\Join\qo}(F_j, G_b)$. 
This implies that every flag of $\po\Join\qo$ has exactly $n+m+3$ elements, including $(F_{-1},G_{-1})$ and $(F_n, G_m)$.

We now turn our attention to show that $\po\Join\qo$ is strongly connected. Consider a section $ (F,G)/(f,g)$ of $\po\Join\qo$ and let $(H,K), (h,k)$ be two proper elements of $(F,G)/(f,g)$. 
Then $f\leq h,H \leq F$ and $g\leq k,K\leq G$. By the strong connectivity of $\po$ and $\qo$, there exist sequences 
$$ h=H^0, H^1, \dots, H^u=H$$ 
and 
$$k=K^0, K^1, \dots , K^v=K$$ 
of elements of $F/f$ and $G/g$, respectively, such that consecutive elements of each sequence are incident and with
$f\leq H^i\leq  F$, $g \leq K^j\leq  G$, for all $i=0, \dots u$ and $j=0, \dots v$.
Without loss of generality we may assume that $u\leq v$. 
Hence, the sequence
$$(h,k)=(H^0,K^0), (H^1,K^0), (H^1,K^1), (H^2,K^1), \dots (H^u, K^u), (H^u, K^{u+1}), \dots, (H^u, K^v)=(H,K)$$
is such that any two consecutive elements are incident and  are all proper faces of the section $ (F,G)/(f,g)$ of $\po\Join\qo$.
Hence $ (F,G)/(f,g)$ is connected and therefore $\po\Join\qo$ is strongly connected.

Finally, we show that the join product $\po\Join\qo$ satisfies the diamond condition. 
Let $(F,G),(f,g) \in \po\Join\qo$ be such that $(f,g)\leq_{\po\Join\qo} (F,G)$ and $$\rank_{\po\Join\qo}(F,G) - \rank_{\po\Join\qo}(f,g) = 2.$$ 
Then  $$(\rank_\po(F)-\rank_\po(f)) + (\rank_\qo(G)-\rank_\qo(g))=2;$$  since we have that $f\leq F$ and $g\leq G$,  the following possibilities arise:
\begin{itemize}
\item $\rank_\po(F)= \rank_\po(f)$ and $ \ \rank_\qo(G)-\rank_\qo(g)=2$; 
\item $\rank_\po(F)- \rank_\po(f) =1$ and $\rank_\qo(G)-\rank_\qo(g)=1$;
\item $\rank_\po(F) - \rank_\po(f) =2$  and $ \rank_\qo(G)=\rank_\qo(g)$. 
\end{itemize}
Note that the first and the last case are symmetric, so it suffices to consider one of them.
In the first case, 
$f=F$ and, by the diamond condition of $\qo$, there are exactly two elements $H_1, H_2$ such that $g<H_1, H_2< G$. Therefore the only elements between $(f,g)$ and $(F,G)$ are $(f,H_1)$ and $(f,H_2)$.
In the second case 
the only two faces between $(f,g)$ and $(F,G)$ are $(f,G)$ and $(F,g)$. 
Therefore $\po\Join\qo$ satisfies the diamond condition, and $\po \Join \qo$ is an abstract polytope.
\end{proof}

The join product on abstract polytopes coincides with the join of two convex polytopes. 
To see this one just needs to note that the $1$-skeleton of $\po \Join \qo$ consists of the union of the $1$-skeleton of $\po$ and the $1$-skeleton of $\qo$, together with all the edges from vertices of $\po$ to vertices of $\qo$.
The most common example of this product is a pyramid over a polygon: if $v$ is a vertex (or a $0$-polytope) and $\po$ is an $n$-gon (a $2$-polytope with $n$ vertices), then $v \Join \po$ is simply the pyramid over the $n$-gon. 
Another common example is to consider two edges $e_1$ and $e_2$ (or line segments), and the join of them: $e_1\Join e_2$ is a tetrahedron.

Suppose for a moment that we were to regard the {\em empty polytope} $\emptyset$ as an abstract polytope of rank $-1$. 
Then, the join of $\emptyset$ with an $n$-polytope $\po$ would simply be the set $\{ (\emptyset, F) \mid F \in \po\}$, and the order will be inherited by that of $\po$. 
It should be then clear that  $\emptyset \Join \po \cong \po$. 
Conversely, if $\qo$ is a $m$-polytope such that $\po \Join \qo \cong \po$ for every polytope $\po$, then $m=-1$. Therefore we shall say that the only {\em trivial polytope with respect to the join product} is the empty polytope.

If now we consider $v$ to be a $0$-polytope (that is, a vertex), then $$v \Join \po = \{ (\emptyset, F) \mid F\in\po\} \cup \{ (v, F) \mid F\in\po\}.$$
That is, the join product of $v$ with a polytope $\po$ gives us two copies of $\po$. However, the rank of an element of the type $(\emptyset, F)$ is $\rank_\po(F)$, while one of the type $(v, F)$ is $\rank_\po(F) +1$, so the two copies of $\po$ are at ``different levels''. We further note that $(\emptyset, F) \leq (v,G)$ if and only if $F\leq_\po G$ (see Figure~\ref{pyramid}).
\begin{figure}[htbp]
\begin{center}
\includegraphics[width=4.8cm]{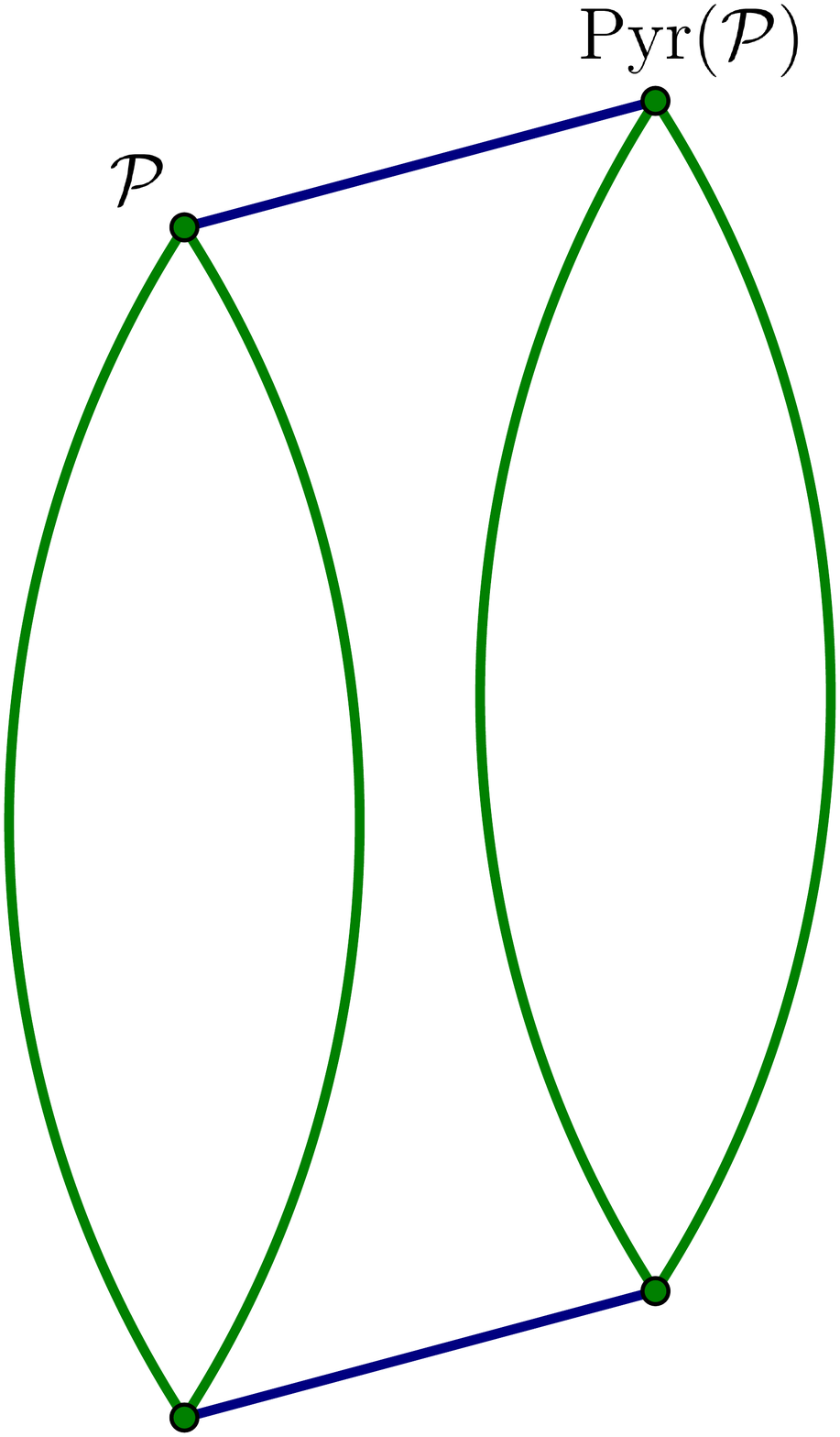}
\caption{A sketch of the Hasse diagram of a pyramid over a polytope $\po$.}
\label{pyramid}
\end{center}
\end{figure}

The $(n+1)$-polytope $v\Join\po$ is called the {\em pyramid} of $\po$ and we shall denote it by $\Pyr(\po)$. It is then straightforward to see that if $\po$ is a 2-polytope (or a polygon), then $\Pyr(\po)$ is simply the pyramid over $\po$. Furthermore, $$\Pyr(\Pyr(\dots \Pyr(v) \dots )) = \Pyr^k(v)$$ is the $(k-1)$-simplex, which is a regular polytope (see Figure~\ref{JoinPowersofaPoint}).

Note that the join product interacts nicely with the dual operation. If $\delta$ and $\omega$ are dualities from  $\po$ to $ \po^*$ and $ \qo$ to $\qo^*$, respectively, then 
$$(\delta, \omega) : \po \Join \qo \to \po^* \Join \qo^*$$  
sending $(F, G)$ to $(F\delta, G\omega)$ is a bijection between $\po \Join \qo$ and $\po^* \Join \qo^*$ such that $(F,G) \leq_{\po\Join\qo} (H,K)$ if and only if $(H\delta, K\omega) \leq_{\po^* \Join \qo^*} (F\delta, G\omega)$. That is, $(\delta, \omega)$ is a duality from $\po\Join\qo$ to its dual, implying that $$(\po \Join \qo)^* \cong \po^* \Join \qo^*.$$
In particular if both $\po$ and $\qo$ are self-dual polytopes, then so is $\po\Join\qo$.

\subsection{Cartesian product and direct sum}

The cartesian product of two abstract polytopes is the natural product when thinking on the geometry: it generalises the cartesian product of two convex polytopes. 
The direct sum can (and will) be defined in terms of the cartesian product and dual polytopes.

Given two posets $\po$ and $\qo$, with minimum faces $F_{-1}$ and $G_{-1}$, respectively, the {\em cartesian product} of $\po$ and $\qo$, denoted $\po \times \qo$, is defined as the set
\begin{eqnarray}
\label{setcart} 
\ \ \po \times \qo = \{(F,G)\in \po*\qo  \mid  \rank_\po(F), \rank_\qo(G) \geq 0 \} \cup \{(F_{-1}, G_{-1})\}, 
 \end{eqnarray}
 where the order is given by 
 \begin{eqnarray}
 \label{ordercart}
 (F, G) \leq_{\po \times \qo} (F', G') \ \mathrm{ if} \ \mathrm{ and} \  \mathrm{ only} \ \mathrm{ if}  \ F \leq_{\po} F' \ \mathrm{ and } \ G \leq_{\qo} G'.
  \end{eqnarray}

Note that the cartesian product of two polytopes $\po$ and $\qo$, as a set, is a subset of the $\po \Join \qo$, the join of $\po$ and $\qo$.
 Hence, it follows at once that $\po \times \qo$ is a poset.
The rank function on $\po \times \qo$ is defined in a different way as for the join product: given a face $(F,G)\in\po\times\qo$, with $\rank{\po}(F), \rank_\qo(G)\geq 0$, we define the rank of $(F,G)$ as, 
$$\rank_{\po \times \qo}(F,G) = \rank_{\po}(F) + \rank_\qo(G);$$ 
and we define the rank of $(F_{-1}, G_{-1})$ to be $-1$.
Hence, if $\po$ is an $n$-polytope and $\qo$ is an $m$-polytope, then $\rank_{\po \times \qo}$ is a function from $\po \times \qo$ to the set $\{-1, 0, \dots, n+m\}$. 
In contrast with the join product, we no longer consider the empty set to be a rank $-1$ polytope, as if we did, we would only have that the cartesian product of any polytope with the empty set is the empty set again. So the product is of no interest.

\begin{prop}
\label{CartesianPolytope}
Let $\po$ and $\qo$ be two polytopes of ranks $n$ and $m$, respectively. Then $\po \times \qo$  is a polytope of rank $n+m$.
\end{prop}

\begin{proof}
As pointed out above, $\po \times \qo$ is a poset. Clearly, $(F_{-1}, G_{-1})$ is its minimal face and, if $F_n$ and $G_m$ denote the maximal faces of $\po$ and $\qo$, respectively, then $(F_n, G_m)$ is the maximal face of $\po \times \qo$.

To see that all the flags of $\po \times \qo$ have the same number of elements, and that $\po \times \qo$ is strongly connected, one can simply adapt the proofs given in the previous section for $\po \Join \qo$. Alternatively, one can think of $\po \times \qo$ as a subset of $\po \Join \qo$ and use this contention to obtain the two properties.

Hence, one only needs to see that $\po \times \qo$ satisfies the diamond condition.

Let $(F,G),(f,g) \in \po \times \qo$ such that $(f,g)\leq_{\po \times \qo} (F,G)$ and $$\rank_{\po\times\qo}(F,G) - \rank_{\po\times\qo}(f,g) = 2.$$ Note that if $(f,g) \neq (F_{-1}, G_{-1})$, the result holds as it did for $\po \Join \qo$.
Hence, without loss of generality we may assume that $(f,g) = (F_{-1}, G_{-1})$. 
This immediately implies that both $F$ and $G$ are proper faces of $\po$ and $\qo$, respectively, and that $\rank_{\po\times\qo}(F,G)=1$. 
If $(H,K) \in \po\times \qo$ is such that $(F_{-1},G_{-1}) < (H,K) < (F,G)$, then $0=\rank_{\po\times\qo}(H,K) = \rank_\po(H) + rank_\qo(K)$. 
Both $H$ and $K$ are proper faces of $\po$ and $\qo$, respectively, and therefore $\rank_\po(H) = rank_\qo(K) =0$.
Since $\rank_{\po\times\qo}(F,G)=1$, then  $\rank_\po(F) + rank_\qo(G) =1$, which in turns implies that either $\rank_\po(F)=1$ and $rank_\qo(G) =0$ or $\rank_\po(F)=0$ and $rank_\qo(G) =1$.
In the first case, by the diamond condition of $\po$ we have that there exist two $0$-faces $H_1, H_2$ such that $F_{-1}<H_1, H_2 < F$. This implies that $(H,K) = (H_1, G)$ or $(H,K) = (H_2, G)$.
The second case is similar and hence the 
diamond condition is satisfied.
\end{proof}

The cartesian product of an edge with a polygon is precisely the prism over the polygon and the cartesian product of an edge with any polytope $\po$ is the prism over $\po$.

The only {\em trivial polytope with respect to  the cartesian product} is the $0$-polytope $v$. It is straightforward to see that for any polytope $\po$, $\po \times v \cong \po$, as the only $0$-face of $v$ is $v$ itself. And conversely, if $\qo$ is a polytope such that $\qo \times \po \cong \po$ for any polytope $\po$, then by considering the rank of $\qo \times \po$ one deduces that the rank of $\qo$ is zero and hence $\qo \cong v$.

One interesting example for the cartesian product is to consider a $1$-polytope $e$ (that is, an edge). Let $v_1$ and $v_2$ be the two $0$-faces of $e$. Then, given an $n$-polytope $\po$,
$$ e \times \po = \{(v_1, F) \mid F\in\po\} \cup \{(v_2, F) \mid F\in\po\} \cup \{(e,F)\mid F\in\po, rank_\po(F)\geq 0 \}.$$ In this case, $e \times \po$ has two isomrphic copies of $\po$ (at the same ``level''), and a third copy of $\po$ with the minimum removed, at one level higher (see Figure~\ref{prismP}). We note further that while $(v_i,F) \leq (e, G)$ whenever $F\leq G$, for $i=1,2$,  two faces of the type $(v_1, F)$ and $(v_2,G)$ can never be incident.
\begin{figure}[htbp]
\begin{center}
\includegraphics[width=8.5cm]{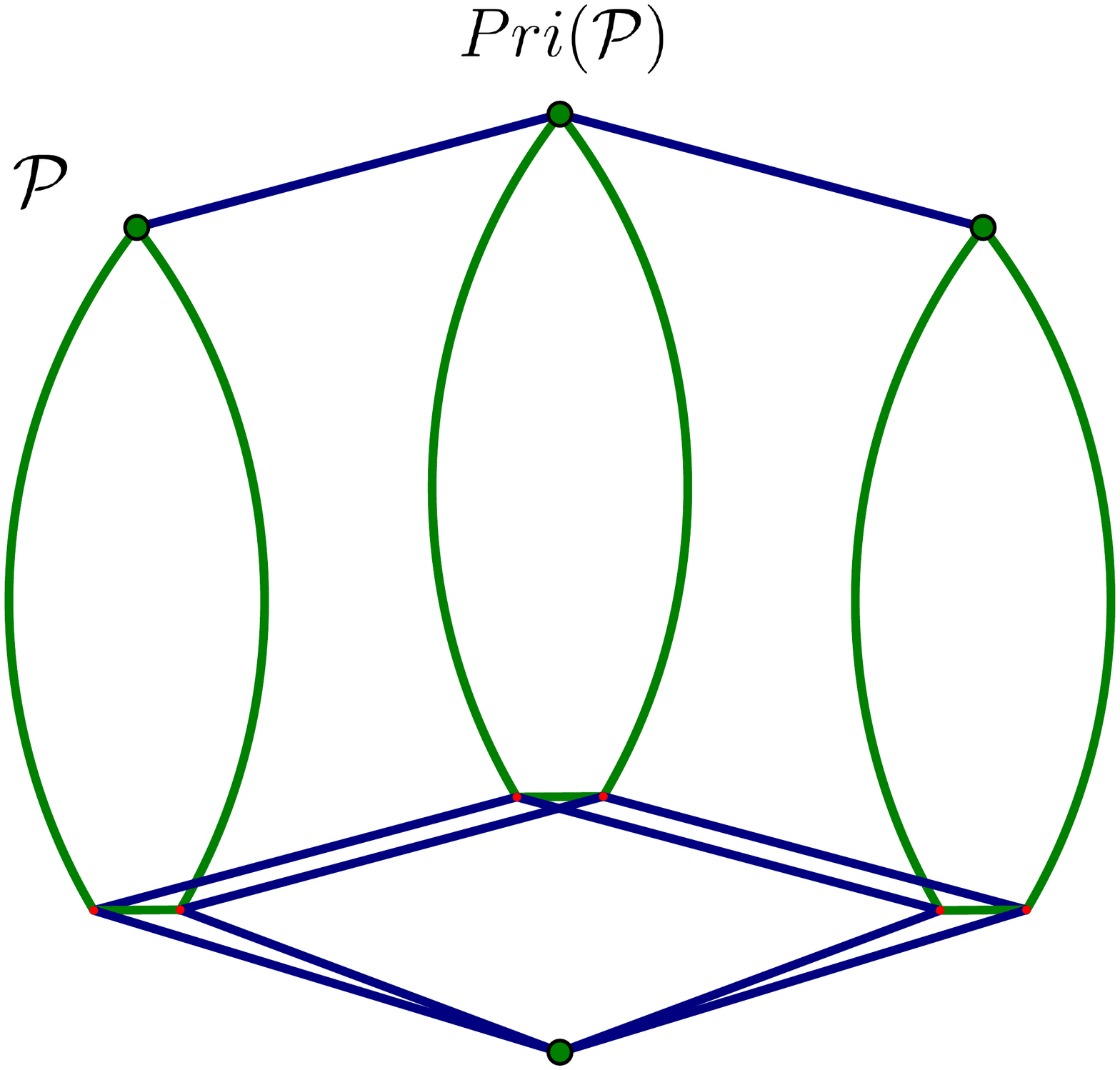}
\caption{Sketch of the Hasse diagram of a prism over a polytope $\po$.}
\label{prismP}
\end{center}
\end{figure}

The $(n+1)$ polytope $e \times \po$ is called the {\em prism} over $\po$ and shall be denoted by $Pri(\po)$. Hence, $Pri (Pri (\dots Pri(e) \dots)) = Pri^d(e)$ is the $d$-cube, which is a regular polytope.

The direct sum of a segment and a polygon is the bipyramid of the polygon. In the introduction we gave a definition of the direct sum of two convex polytopes. The direct sum of two convex polytopes can be described, using duality, in terms of a cartesian product. In fact, we have that for convex polytopes, $\po \oplus \qo := (\po^* \times \qo^*)^*,$ (see for example \cite[Lemma 2.4]{bremner}) where $\po^*$ denotes the polar dual  of $\po$ .

Hence, given two abstract polytopes $\po$ and $\qo$, we define the {\em direct sum} of $\po$ and $\qo$, denoted by $\po \oplus \qo$, simply as
$$\po \oplus \qo := (\po^* \times \qo^*)^*.$$  It is straightforward to see that if $F_n$ and $G_m$ are the maximal elements of $\po$ and $\qo$, respectively, then we have that
\begin{eqnarray}
\label{setdir}
 \po \oplus \qo = \{(F,G)\in\po*\qo  \mid   \rank_\po(F)<n \ \mathrm{and} \ \rank_\qo(G) < m \} \cup \{(F_{n}, G_{m})\},
 \end{eqnarray}
 where the order is given by 
 \begin{eqnarray}
 \label{orderdir}
 (F, G) \leq_{\po \oplus \qo} (F', G') \ \mathrm{ if} \ \mathrm{ and} \  \mathrm{ only} \ \mathrm{ if}  \ F \leq_{\po} F' \ \mathrm{ and } \ G \leq_{\qo} G'.
  \end{eqnarray}
  
  An immediate corollary of Proposition~\ref{CartesianPolytope} is the following result.
  
\begin{coro}
Let $\po$ and $\qo$ be two polytopes of ranks $n$ and $m$, respectively. Then $\po \oplus \qo$ (as defined in (\ref{setdir}) and (\ref{orderdir})) is a polytope of rank $n+m$.
\end{coro}

Similarly as above, the only  {\em trivial polytope with respect to  the cartesian product} is the $0$-polytope $v$. In fact, for any polytope $\po$, $\po \oplus v = (\po^* \times v^*)^* = (\po^*)^* \cong \po$, as $v^*$ is $v$ itself. And conversely, if $\qo$ is a polytope such that $\qo \oplus \po \cong \po$ for any polytope $\po$, then by considering the rank of $\qo \oplus \po$ one deduces that the rank of $\qo$ is zero and hence $\qo \cong v$. 

Given an $n$-polytope $\po$, the $(n+1)$-polytope $e \oplus \po$ is called the {\em bipyramid} over $\po$, and shall be deonted by $Bpy(\po)$. In this case, $Bpy ( Bpy(\dots Bpy(e)\dots))=:Byp^d(e)$ is the $d$-cross-polytope, which is a regular (convex) polytope, that is dual to the $d$-cube.

\subsection{Topological product}

The last product that we consider in this paper, does not have a convex analogue. The name is given with the following example in mind: the topological product of two polygons (homeomorphic to circles $\mathbb{S}^1$) gives us a map on the torus (the product of $\mathbb{S}^1 \times \mathbb{S}^1$). 

Given an $n$-polytope $\po$ with minimum element $F_{-1}$ and maximum element $F_n$, and an $m$-polytope $\qo$ with minimal element $G_{-1}$ and maximal element $G_m$, the {\em topological product} of $\po$ and $\qo$, denoted by $\po \square \qo$, is defined as
\begin{eqnarray}\label{settopo}
\ \ \ \ \po \square \qo = \{ (F,G) \in \po * \qo \mid  0\leq \rank_\po(F)<n,\ 0\leq\rank_\qo(G)<m\} \cup \{(F_{-1}, G_{-1}),(F_{n}, G_{m})\},
\end{eqnarray}
 where the order is given by 
 \begin{eqnarray}
 \label{ordertopo}
 (F, G) \leq_{\po \square \qo} (F', G') \ \mathrm{ if} \ \mathrm{ and} \  \mathrm{ only} \ \mathrm{ if}  \ F \leq_{\po} F' \ \mathrm{ and } \ G \leq_{\qo} G'.
  \end{eqnarray}
  
  Here, we say that the rank of the faces $(F_{-1}, G_{-1}),(F_{n}, G_{m})\in\po\square\qo$ are $-1$ and $n+m-1$, respectively, and given $(F,G)\in\po\square\qo$ with $0\leq \rank_\po(F)<n, \ 0\leq\rank_\qo(G)<m$, then $$\rank_{\po\square\qo}(F,G)=\rank_\po (F) +\rank_\qo (G).$$
  
Note that if $\po$ has rank $0$, then $\po \square \qo \cong \po$ for every polytope $\qo$. Moreover, if $\po$ has rank $1$, and $\qo$ has rank at least $1$, then $\po \square \qo$ is not connected, implying that it is not a polytope. However, using a similar proof as that of Proposition~\ref{CartesianPolytope}, we have the following proposition.

\begin{prop}
Let $\po$ and $\qo$ be two polytopes of ranks $n$ and $m$, respectively, with $n,m\geq 2$. Then $\po \square \qo$ (as defined in (\ref{settopo}) and (\ref{ordertopo})) is a polytope of rank $n+m-1$.
\end{prop}

There are no trivial polytopes for the topological product. If $\po_1, \dots \po_d$ is a collection of $2$-polytopes, then $\square_{i=1}^d \po_i$ is a $d$-torus tessellated by $d$-cubes. In particular if every $\po_i$ is isomorphic to a $p$-gon, then $\square_{i=1}^d \po_i$  is the regular $(d+1)$-polytope $\{4,3^{d-1}\}_{(a,0,\dots, 0)}$ (see \cite{arp}).

\section{Unique factorization theorems for products of polytopes}\label{sec:fact}

The purpose of this section is to show that, for any of the four products described in the previous section, any polytope can be factored in a unique way (up to isomorphism) as the product of prime polytopes. As the proofs of this result for each of the four products are very similar, we shall first view all four products as cardinal products of posets and show some results for such products. 

We start by noticing that given a $n$-polytope $\po$ with minimum element $F_{-1}$ and maximum element $F_n$, and an $m$-polytope $\qo$ with minimal element $G_{-1}$ and maximal element $G_m$, we have that
\begin{eqnarray*}
\po \Join \qo &=& \po * \qo; \\
\po \times \qo &=& (\po \setminus \{F_{-1} \}) * (\qo \setminus \{G_{-1} \}) \cup \{(F_{-1},G_{-1})\}; \\
\po \oplus \qo &=& (\po \setminus \{F_{n} \}) * (\qo \setminus \{G_{m} \}) \cup \{(F_{n},G_{m})\}; \\
\po \ \square  \ \qo &=& (\po \setminus \{F_{-1}, F_n \}) * (\qo \setminus \{G_{-1}, G_m \}) \cup \{(F_{-1},G_{-1}), (F_n, G_m)\} .
\end{eqnarray*}

In other words,
\begin{eqnarray*}
\po \Join \qo &=& \po * \qo; \\
\po \times \qo \setminus \{(F_{-1},G_{-1})\} &=& (\po \setminus \{F_{-1} \}) * (\qo \setminus \{G_{-1} \}); \\
\po \oplus \qo \setminus \{(F_{n},G_{m})\}&=& (\po \setminus \{F_{n} \}) * (\qo \setminus \{G_{m} \}); \\
\po \ \square  \ \qo \setminus \{(F_{-1},G_{-1}), (F_n, G_m)\} &=& (\po \setminus \{F_{-1}, F_n \}) * (\qo \setminus \{G_{-1}, G_m \});
\end{eqnarray*}
which says that maybe with exception of the minimum and maximum faces, the four products of polytopes can be seen as cardinal products of posets.

Theorem~\ref{uniquefactposet} then implies that each of the four products of polytopes has a unique prime factorisation in terms of posets, however, we want to show that the prime factors are also abstract polytopes. 

The following lemma is straightforward.
\begin{lemma}
\label{factminmax}
Let $\po$ be poset with minimum element (resp. maximumm) and suppose there exist posets $\qo$ and $ \ko$  such that $\po = \qo * \ko$. Then $\qo$ has a minimum element (resp. maximum).
\end{lemma}

%

By the commutativity of the product, the above lemma implies that also $\ko$ has a minimum and/or maximum, whenever $\po$ has it too.

In what follows, for an $n$-polytope $\po$, $\breve{\po}$ will be denoting a polytope $\po$ without its minimum and/or maximum elements. Hence, $\breve{\po}$ satisfies the following properties.
\begin{enumerate}
\item[P1.] $\breve{\po}$  is a poset with a rank function, in which all the maximal chains have the same number of elements.
\item[P2.] $\breve{\po}$ satisfies the diamond condition for $i = 1, \dots, n-2$, and for every face of rank $n-2$ (1, resp.), there are exactly two $(n-1)$-faces ($0$-faces, resp.) incident to it.
\item[P3.] $H(\breve\po)$ is a weakly connected digraph, and every open interval of $\breve{\po}$ either has two elements or it is also connected.
\end{enumerate}

Note that if a poset $\qo$ satisfies the three above properties, we can extend $\qo$ by defining a minimum and a maximum elements of the order, and then the resulting new poset is indeed a polytope. 
In the next lemmas we shall establish that the factors of a factorable poset $\breve{\po}$ have the properties P1, P2 and P3.

\begin{lemma}
\label{factrankflag}
Let $\po$ be an $n$-polytope and suppose there exist posets $\qo$ and $ \ko$  such that $\breve\po = \qo * \ko$. Then $\qo$ (and therefore $\ko$) has a rank function. Furthermore, all the flags of $\qo$ (and therefore of $\ko$) have the same number of elements.

\end{lemma}

\begin{proof}
We assume that $\breve\po$ does not have a maximum and minimum elements. The arguments are similar in the case it has one of them. Hence, $\breve\po$ has a rank function with range $\{0,1, \dots, n-1\}$.
Fix $(Q,K)$ to be a maximal face of $\breve\po$, thus, $(Q,K)$ has rank $n-1$.

Consider
$\tilde\qo:= \{ (x, K) \mid x \in \qo\}.$
Then $\tilde\qo \subset \breve\po$ and $\tilde\qo \cong \qo$.

Note that the maximality of $(Q,K)$ implies that for every $x\in\tilde\qo$, we have that $\rank_\po (x,K) \leq \rank_\po (Q,K)= n-1$. Hence, $Q$ is a maximal element of $\qo$ (though most likely it is not maximum).
Moreover, as $\tilde\qo \subset \breve\po$, then there exists $q \in \qo$ such that $q<Q $ and $\rank_\po (q,K) \leq \rank_\po (x,K)$ for every $x \in \qo$. 
Thus, $q$ is a minimal (but not minimum) element of $\qo$.

Let $a:=\rank_\po (q,K)$. 
We now show that if $y$ is another minimal element of $\qo$, then $\rank_\po (y,K)=a$.
Since $(q,K)<_{\breve\po}(Q,K)$, then we can complete $\{(q,K), (Q,K)\}$ to a maximal chain $\Phi$ of $\breve\po$.
Hence, the minimal element $(q,k)$ of $\Phi$ has rank $0$ in $\breve\po$ and $k$ is a minimal element of $\ko$.
Consider the set $\Phi_<$ of all that elements of $\Phi$ that have rank less or equal to $a$. 
Since $q$ is minimal in $\qo$, then the first coordinate of all such elements is in matter of fact $q$.
Thus, the set $\Lambda$ consisting of the second coordinates of $\Phi_<$ is a maximal chain of $\ko$.

Now, let $y\in\qo$ be a minimal element (of $\qo$). 
We can complete $(y,K)$ to a maximal chain $\Psi$ of $\breve\po$ in such a way that all the elements of $\Psi$ with rank less than $b:=\rank_\po (y,K)$ are of the form $(y,x)$ with $x\in \Lambda$.
Since $\po$ is a polytope, then all the maximal chains of $\breve\po$ have $n$ elements. 
Hence, both $\Phi$ and $\Psi$ have $n$ elements.
As the number of elements in $\Phi$ of rank less than $a$ equals the number of elements in $\Psi$ of rank less than $b$, then the number of of elements in $\Phi$ of rank greater than $a$ equals the number of elements in $\Psi$ of rank greater than $b$, implying that $a=b$.

Since $\po$ is an $n$-polytope, then there exists a rank function $\rank_\po : \po \to \{-1, \dots, n\}$. 
Hence, $$\rank_\po\mid_{\tilde\qo} \ : \tilde\qo \to \{ a, \dots, n-1\}.$$
Thus, by defining for each $x \in \qo$, 
\begin{eqnarray}
\label{rankQ}
\rank_\qo (x) := \rank_\po (x,K) - a,
\end{eqnarray} we obtain a rank function from $\qo$ to the set $\{ 0, \dots, n-1-a\}$, and the first part of the lemma has been stablished.
%
%

Note now that any maximal chain of $\qo$ must have at most $n-1-a+1=n-a$ faces. 
Let $\Phi$ be a maximal chain of $\qo$, and suppose $\Phi$ has less than $n-a$ elements.
Let $y,z \in \qo$ be the minimal and maximal elements of $\Phi$, respectively.
We have shown that all minimal elements of $\qo$ have the same rank and one can similarly show that all maximal elements also have the same rank. Hence, $y,z$ have ranks zero and $n-a$, respectively.
So let $c \in \{1, \dots, n-a\}$ be such that there is no element in $\Phi$ of rank $c$ and that $c$ is minimal in that sense. Then, there exists $w \in \qo$ such that $\rank_\qo(w)=c-1$.

Let $\tilde\Phi := \{ (x, K) \mid x\in \Phi\}$. Then $\tilde\Phi$ is a maximal chain of $\tilde\qo$. Extend $\tilde\Phi$ to a flag $\Psi$ of $\breve\po$, and consider its faces $\Psi_{c+a-1}$ and $\Phi_{c+a}$ (of ranks $c+a-1$ and $c+a$, respectively). Since $\rank_\qo(w)=c-1$, then $\Psi_{c+a-1} = (w,K)$. 
As there exists no element of $\qo$ of rank $c$, there exists no element of $\tilde\qo$ of rank $c+a$, and therefore $\Phi_{c+a}$ is not an element of $\tilde\qo$. 
This implies that there exists $G \in \qo$ and $H \in \ko$ with $H < K$ such that $\Phi_{c+a} = (G,H)$. But since $\Psi$ is a flag and $(w,K), (z,K) \in \Psi$, then $(w,K) \leq (G, H) \leq (z,K)$. This immediately implies that $H=K$, which is a contradiction.

Therefore for every $c \in \{0, \dots, n-a-1\}$ there is an element of $\Phi$ of rank $c$ and thus all flags of $\qo$ have the same number of elements, namely $n-a$.
\end{proof}

\begin{lemma}
\label{factconn}
Let $\po$ be an $n$-polytope and suppose there exist posets $\qo$ and $ \ko$  such that $\breve\po = \qo * \ko$. Then $\qo$ (and therefore $\ko$) is connected and so is every  interval of it.
\end{lemma}

\begin{proof}
We start by showing that $\qo$ is connected. Suppose otherwise. Then $H(\qo)$ is a disconnected digraph and hence $H(\breve\po) \cong H(\qo)*H(\ko)$ is disconnected. This in turns implies that $\breve\po$ is disconnected, which is a contradiction.

We shall now see that every interval of $\qo$ is in fact isomorphic to an interval of $\breve\po$. The proposition will follow then at once.
In fact, given $F, G \in \qo$ with $F<G$, the intervals $\{ H \in \qo \mid F< H \}$, $\{ H \in \qo \mid F> H \}$ and $\{ H \in \qo \mid F< H < G \}$ are respectively isomorphic to the intervals of $\{ (x,M) \in \breve\po \mid (F, M)< (x,M)\}$, $\{ (x,M) \in \breve\po \mid (F, M)> (x,M)\}$ and $\{ (x,M) \in \breve\po \mid (F,M)< (x,M) < (G,M) \}$, where $M$ is a fixed element of $\ko$. Thus every interval of $\qo$ is connected.
\end{proof}

 \begin{lemma}
 \label{factdiam}
Let $\po$ be an $n$-polytope and suppose there exist posets $\qo$ and $ \ko$  such that $\breve\po = \qo * \ko$. Let $m \in \mathbb{Z}$ be such that $\rank_\qo : \qo \to \{0, \dots, m\}$ is the rank function defined in~(\ref{rankQ}). Then,
\begin{enumerate}
\item[a)] If $F, G \in \qo$ are such that $F \leq G$ with $\rank_\qo(G) - \rank_\qo(F) = 2$, then there are exactly two faces $H \in \qo$ such that $F < H < G$.
\item[b)] If $F \in \qo$ is such that $\rank_\qo(F) = m-1$, then there are exactly two faces $H \in \qo$ such that $F < H $.
\item[c)] If $G \in \qo$ is such that $\rank_\qo(G) = 1$, then there are exactly two faces $H \in \qo$ such that $H < G $.
\end{enumerate}
\end{lemma}

\begin{proof}
We start by showing part $a)$. Let $M$ be an element of $\ko$. Then, $(F,M) \leq (G,M)$ and $\rank_\po(G,M) - \rank_\po(F,M) = 2$. By the diamond condition of $\po$ there exist exactly two elements $x \in \breve\po$ such that $(F,K) < x < (G,K)$. By the definition of the cardinal product $\qo*\ko$, the second coordinate of $x$ must be $K$. Hence, part $a)$ of the proposition follows.
Parts $b)$ and $c)$ follow in a similar fashion watching for the details. In fact for part $b)$, we must set $K=M$ a maximal element of $\ko$, while for part $c)$, $K=m$ a minimal element of $\ko$.
\end{proof}

Using Lemmas~\ref{factminmax},~\ref{factrankflag},~\ref{factconn} and~\ref{factdiam} we can now establish the following theorem.

\begin{theorem}
\label{factpoly}
Let $\po$ be an $n$-polytope with minimum element $F_{-1}$ and maximum element $F_n$ and  $\qo$ and $\ko$ be posets. Then we have the following.
\begin{enumerate}
\item If $\po = \qo * \ko$, then both $\qo$ and $\ko$ are polytopes.
\item If $ \po \setminus \{F_{-1}\} = \qo * \ko$, then both $\qo$ and $\ko$ have a maximum element and satisfy properties P1, P2 and P3.
\item If $ \po \setminus \{F_{n}\} = \qo * \ko$, then both $\qo$ and $\ko$ have a minimum element and satisfy properties P1, P2 and P3.
\item If $ \po \setminus \{F_{-1}, F_n\} = \qo  *  \ko$, then both $\qo$ and $\ko$ satisfy properties P1, P2 and P3.
\end{enumerate}
\end{theorem}

\begin{coro}
Let $\po$ be an abstract polytope and let $\odot$ denote a product of polytopes (either the join, cartesian or topological product, or the direct sum). Then $\po$ can be uniquely factorised as a $\odot$-product polytopes that are prime with respect to the product $\odot$.
\end{coro}

\begin{proof}
The corollary follows from Theorems~\ref{uniquefactposet} and~\ref{factpoly}.
\end{proof}

\section{The flags of a product}\label{sec:flags}

In Sections~\ref{sec:auto} and~\ref{sec:mono} we shall deal with the groups and orbits of products. To study these groups it shall prove very helpful to have a better understanding of the structure of the flags of a product. That is the 
purpose of this section.

We start by analyzing the join product.
Let $\po$ be an $(n-1)$-polytope and
suppose that $\po = \qo_1 \Join \qo_2 \Join \dots \Join \qo_r$, for some polytopes $\qo_1\dots\qo_r$ such that $\qo_i$ has rank $n_i-1$. 
This implies that $$n=n_1+ n_2 + \dots + n_r.$$
Without loss of generality we may assume that $n_i\geq 1$.

Let $\Phi$ be a flag of $\po$. 
Then $\Phi = \{\Phi_{-1}, \Phi_0, \Phi_1, \dots, \Phi_{n-1}\}$, where $\Phi_i$ has rank $i$. 
Since $\po$ is a product, then for each $i\in\{-1, \dots n-1\}$ there exist $F^j_i \in \qo_j$ such that $$\Phi_i = (F_i^1, F_i^2, \dots, F_i^r).$$
By definition of the join product, we have that for each $j = 1,2, \dots, r$, 
$$F^j_{-1}\leq F^j_0 \leq F^j_1 \leq \dots \leq F^j_{n-1},$$
where $F^j_{-1}$ and $F^j_{n-1}$ are the minimum and maximum elements, respectively, of $\qo_j$.
Note that many of the $F^j_i$ are repeated in the above sequence, as otherwise $\po$ would be just a trivial product.
That means that the set $\{F^j_{-1}, F^j_0 , F^j_1 , \dots , F^j_{n-1}\}$ has cardinality $n_i+1$ and, after erasing the repeated faces, it is a flag of $\qo_j$.
Call $\Psi^{(j)}$ such flag.

Now, since  $\Phi = \{\Phi_{-1}, \Phi_0, \Phi_1, \dots, \Phi_{n-1}\}$ is a flag of $\po$, then 
$$\rank(\Phi_{i}) - \rank(\Phi_{i-1}) = 1,$$ 
for each $i$.
Again, by the definition of the join product, for each $i$, $\Phi_{i-1} = (F_{i-1}^1, F_{i-1}^2, \dots, F_{i-1}^r)$ and $\Phi_i = (F_i^1, F_i^2, \dots, F_i^r)$   differ in exactly one entry.
Denote by $a^{(\Phi)}_i$ such entry.
In other words, $a^{(\Phi)}_i=j \in \{1, \dots, r\}$ if and only $\Phi_{i-1}$ and $\Phi_i$ differ in their $j$ entry.
Thus, we can naturally identify each flag $\Phi$ of $\po$ with the ordered pair $(\{\Psi^{(1)}, \Psi^{(2)}, \dots, \Psi^{(r)}\}, a)$, where each $\Psi^{(i)}$ is the flag of $\qo_i$ described above and $a=\{a^{(\Phi)}_0, a^{(\Phi)}_1, \dots, a^{(\Phi)}_{n-1}\}$. 
Clearly, two different flags of $\po$ define different ordered pairs.

Note further that for each flag $\Phi \in \fl(\po)$, the sequence $a^{(\Phi)}_0, a^{(\Phi)}_1, \dots, a^{(\Phi)}_{n-1}$ has exactly $n_j$ times the integer $j$, for each $j\in \{1,\dots, r\}$.

Let $\mathcal A$ be the set of all ordered $n$-tuples $a=\{a_0,a_1, \dots, a_{n-1}\}$ with $a_j\in \{1, \dots, r\}$ and such that each $j \in \{1, \dots, r\}$ appears exactly $n_j$ times in $a$.
Given an ordered pair $(\{ \Psi^{(1)}, \dots, \Psi^{(r)}\}, a)$, where $\Psi^{(j)}$ is a flag of $\qo_j$ and $a \in {\mathcal A}$, we can define the flag $\Phi = \{\Phi_{-1}, \Phi_0, \dots, \Phi_{n-1}\}$ of $\po$ as follows. 
The minimum face of $\Phi$, $\Phi_{-1}$ is the $r$-tuple $(F_{-1}^1, F_{-1}^2, \dots F_{-1}^r)$, where each $F_{-1}^j$ is the minimum face of the polytope $\qo_j$. 
Suppose that we have defined the $(i-1)$-face $\Phi_{i-1}=(F_{i-1}^1, F_{i-1}^2, \dots, F_{i-1}^r)$ of $\Phi$, in such a way that $F_{i-1}^j$ is a face of the flag $\Psi^{(j)}$ of $\qo_j$. 
Hence, for each $j \in \{1,\dots, r\}$, the face $F_{i-1}^j$ of $\qo_j$ has some rank, say $t_j$, with $-1\leq t_j \leq n_j$.
The $i$-face $\Phi_{i}$ is the $r$-tuple that coincides with $\Phi_{i-1}$ in all its entries, except in the entry $a_i$. 
The entry $a_i$ of $\Phi_i$ is the face of the flag $\Psi^{(a_i)}$ of rank $t_j +1$. 
Hence, in particular, $\Phi_0$ has all its entries equal the minimum face of the corresponding polytope (all of rank $-1$), except for its $a_0$ entry, which is the $0$-face of the flag $\Psi^{(a_0)}$.

If we now take one of the other three products, the analysis of the flags is very similar. The main differences are in the way the set $\mathcal A$ should be defined for each product and, thus, in the way to construct a flag of the product, given one flag of each factor and an element of $\mathcal A$. 
Alternatively, one can keep $\mathcal A$ fixed and adjust the ranks of the polytopes $\qo_i$ as well as the definition of the vertices or the facets of the product, depending on the product we are dealing with. 
Using similar methods as the ones explained above one can show the following lemma. 

\begin{lemma}
\label{flagbijection}

Let $\qo_1,\qo_2, \dots, \qo_r$ be polytopes and $\odot$ denote one of the four products discussed in Section~\ref{sec:allproducts} (that is, $\odot \in \{ \Join, \times, \oplus, \square\}$).
Let $\po=\qo_1 \odot \qo_2 \odot \dots \odot \qo_r$ and let $\fl$ denote the  set $\fl(\qo_1) \times \fl(\qo_2) \times \dots \times \fl(\qo_r)$.
Then there exists a bijection $\varphi_\po$ between 
$
\fl(\po) $ and $ \fl\times{\mathcal A}$,
where, $\mathcal A$ is the set of ordered $n$-tuples with entries in the set
$ \{1, \dots, r\}$ and such that each $j \in \{1, \dots, r\}$ appears exactly $n_j$ times in $a$, 
$n=n_1+n_2+\dots+n_r$,
and $n_j$ is related to the 
 rank of the polytope $\qo_j$ in the following way: 
\begin{eqnarray*}
\rank \qo_j = \left\{ \begin{array}{ll}
n_j-1 &\mathrm{if} \ \odot=\Join;\\
n_j&\mathrm{if} \ \odot=\times, \oplus;\\
n_j+1 &\mathrm{if} \ \odot=\square.
\end{array}\right.
\end{eqnarray*}
\end{lemma}

\section{Automorphism groups of products}\label{sec:auto}

In this section we turn our attention to the automorphism group of a product of polytopes. Throughout the section, $\odot$ will denote one of the four products discussed in Section~\ref{sec:allproducts} (i.e., $\odot \in \{ \Join, \times, \oplus, \square\})$, and we shall refer to the $\odot$-product simply as the product. Likewise, a prime polytope will be a prime polytope with respect to $\odot$.

Although $\odot$ cannot always be seen as a cardinal product of posets, we note that the automorphism group of a polytope $\po$ coincides with the automorphism group of $\po$ taking away the minimum or  maximum elements or both. Hence, for proposes of computing the automorphism group of a product $\odot$ of polytopes, without loss of generality we may assume that $\odot$ is in fact the cardinal product of posets. 

We shall say that two polytopes $\po$ and $\qo$ are {\em relatively prime} if their (unique) prime factorization does not have any prime polytopes in common. In particular, if both $\po$ and $\qo$ are different prime polytopes, then they are relatively prime. 

In~\cite{duffus} (Corollary 2), Duffus shows that every automorphism $\gamma$ of a product $\po * \qo$ of relatively prime posets is the product of an automorphism $\gamma_\po$ of $\po$ times an automorphism $\gamma_\qo$ of $\qo$. From this fact, we obtain the following proposition.

\begin{prop}
\label{relprime}
Let $\po$ and $\qo$ be two relatively prime polytopes. Then, $\G(\po \odot \qo) \cong \G(\po) \times \G(\qo)$.
\end{prop}

Corollary 2 of~\cite{duffus} also states that if $\qo$ is a prime poset and $m\in \mathbb{N}$, then for any automorphism $\gamma$ of $\po:=\Pi_{i=1}^n \qo$ there exist a permutation $\sigma$ of the set $\{1, \dots, m\}$ and automorphisms $\gamma_1, \dots, \gamma_m$ of $\qo$ such that for every $F=\{F_1, \dots, F_m\} \in \qo$, and every $i \in \{0, \dots, m\}$, the $i$-th coordinate of the element $F\gamma$ is precisely $F_{i\sigma}\gamma_i$ (where $i\sigma$ is precisely the image of $i$ under the permutation $\sigma$). That is, $F\gamma = (F_{1\sigma}\gamma_1, F_{2\sigma}\gamma_2, \dots, F_{m\sigma}\gamma_m)$. 
Moreover, it is clear that an element $\alpha= (\alpha_1, \alpha_2, \dots, \alpha_m) \in \G(\qo) \times \G(\qo) \times \dots \times \G(\qo)$ acts naturally on the elements of $\po$. Namely, $F\alpha = (F_1\alpha_1, \dots F_m\alpha_m)$. From these two facts, one can show the following proposition.

\begin{prop}
\label{power}
Let $\qo$ be a prime polytope and let $\po:= \Pi_{i=1}^m \qo$. Then $\G(\po) \cong \Pi_{i=1}^m \G(\qo) \rtimes S_m$.
\end{prop}

By denoting $\Pi_{i=1}^m \qo$ simply by $\qo^m$ and $\Pi_{i=1}^m \G(\qo)$ as $\G(\qo)^m$, from Propositions~\ref{relprime} and~\ref{power}, we obtain the following corollary, which settles part $c)$ of Theorem A.

\begin{coro}
If $\po =  \qo_1^{m_1} \odot \qo_2^{m_2}\odot \dots \odot \qo_r^{m_r}$, where the $\qo_i$ are distinct prime polytope, then
$$ \Aut(\po) = \Pi_{i=1}^r (\Aut(\qo_i)^{m_i}\rtimes S_{m_i}).$$
\end{coro}

The main interest in the study of abstract polytopes has been the highly symmetric ones, being the regular ones the most studied ones in the last 30 years (see for example \cite{arp}). As one naturally expects, the product of two regular polytopes in general is not a regular polytope anymore. In fact, we shall see that with exception of one family per product, regular polytopes are prime.


Although different products are described in a slightly different way, we can study them all under the same scope, so we let $\odot \in \{\Join, \times, \oplus, \square\}$.
By Lemma~\ref{flagbijection}, there is a bijection between the flags of $\po =  \qo_1 \odot \qo_2 \odot \dots \odot \qo_r$ and $\fl \times  {\mathcal A}$, where $\fl = \fl(\qo_1) \times \fl(\qo_2) \times \dots \times \fl(\qo_r)$ and $\mathcal A$ is 
 the set of all sequences $a_1, \dots a_n$ such that $a_i \in \{1, \dots, r\}$ and that for each $j \in \{1, \dots, r\}$ the integer $j$ appears exactly $n_j$ times in the sequence,
 $n_j$ stands for the rank of $\qo_j$ plus or minus $1$, depending which  product of polytopes we would want to consider and $n-1=n_1+\dots n_r-1$ is the rank of $\po$.

Note that the cardinality of ${\mathcal A}$ is $|{\mathcal A}| = {n \choose n_1}{ {n-n_1} \choose n_2} \dots {n_r \choose n_r}$, and let us denote $\fl \times {\mathcal A}$ by ${\mathcal B}(\po)$.


Suppose now that $r=2$ (that is, $\po = \qo_1 * \qo_2$) and that the posets $\qo_1$ and $\qo_2$ are relatively primes. Let $\gamma \in \G(\po)$. Then, $\gamma = (\gamma_1, \gamma_2)$, where $\gamma_j \in \G(\qo_j)$, and the action of $\gamma$ on an element $(\{ \Phi_1, \Phi_2\}, a)$, of ${\mathcal B}(\po)$ is given by
$$(\{ \Phi_1, \Phi_2\}, a)\gamma = (\{\Phi_1\gamma_1, \Phi_2\gamma_2\}, a).$$

Note that given $a, a' \in {\mathcal A}$ with $a \neq a'$, there is no element of $\G(\po)$ that can send an element of ${\mathcal B}(\po)$ with second coordinate $a$ to one with second coordinate $a'$. Furthermore, in order to have $(\{ \Phi_1, \Phi_2\}, a)$ and $(\{ \Psi_1, \Psi_2\}, a)$ in the same orbit, we need $\Phi_j$ and $\Psi_j$ in the same orbit under $\G(\qo_j)$. 
If for each integer $n_i$, we let
\begin{eqnarray}\label{not:ranks}
n_i^\odot = \left\{ \begin{array}{ll}
n_i-1 &\mathrm{if} \ \odot=\Join;\\
n_i &\mathrm{if} \ \odot=\times, \oplus;\\
n_i+1 &\mathrm{if} \ \odot=\square,\\
\end{array}\right.
\end{eqnarray}
then we have  established the following lemma. 

\begin{lemma}
\label{orbitstworelatprime}
Let $\po$ be an $(n-1)$-polytope, and suppose $\po = \qo_1 \odot \qo_2$ with $\qo_1$ and $\qo_2$ relatively prime with respect to $\odot$. 
Let $n_i^\odot$  (as in (\ref{not:ranks})) denote the rank of $\qo_i$ , and let $k_i$  denote the number of orbits of $\G(\qo_i)$ on $\fl(\qo_i)$. Then the number of orbits of $\fl(\po)$ under the action of $\G(\po)$ is
$k_1k_2{{n_1+n_2} \choose n_2}$.
\end{lemma}


By induction we can now obtain the following corollary

\begin{coro}
\label{coro:orbitsrelatprimes}
Let $\po = \qo_1\odot\qo_2,\odot \dots,\odot \qo_r$ with $\qo_i$ and $\qo_j$ relatively prime for every $i, j \in \{1, \dots, r\}$. 
Let $n_i^\odot$ (as in (\ref{not:ranks})) be the rank of $\qo_i$ and $k_i$  denote the number of orbits of $\G(\qo_i)$ on $\fl(\qo_i)$. 
Then the number of orbits of $\fl(\po)$ under the action of $\G(\po)$ is
$$k_1k_2\dots k_r \frac{(n_1+n_2+\dots+n_r)!}{n_1!n_2!\dots n_r!}.$$
\end{coro}


We now turn our attention to the case when $\po=\qo\odot\dots\odot\qo = \qo^m$, where $\qo$ is a prime polytope with respect to $\odot$ and  $m$ is a natural number. 
In this case,  the action of $\G(\po)=\G(\qo)^m \rtimes S_m$ on the elements of ${\mathcal B}(\po)$ is given by:
$$(\{ \Phi_1, \dots, \Phi_m\}, (a_1,a_2,\dots, a_n))\gamma = (\{ \Phi_{1\sigma}\gamma_1, \dots, \Phi_{m\sigma}\gamma_m\}, (a_1\sigma^{-1}, a_2\sigma^{-1},\dots, a_n\sigma^{-1}).$$

We have seen that $|{\mathcal A}| = {n \choose n_1}{ {n-n_1} \choose n_2} \dots {n_m \choose n_m}$. In this case this means that, if $N^\odot$ is the rank of $\qo$, then $n=MN$ and
$$|{\mathcal A}| = {n \choose N}{ {n-N} \choose N} \dots {N \choose N} = {Nm \choose N}{ {N(m-1)} \choose N} \dots {N \choose N} = \frac{(Nm)!}{(N!)^m}.$$

Note that each $a \in {\mathcal A}$ can be sent by an element of $\G(\po)$ to $m!$ elements. Indeed, each element of $S_m$ acts on the second coordinate of the elements of $\mathcal B(\po)$, and the only element of $S_n$ that fixes a given $a\in{\mathcal A}$ is the identity. 
Moreover, only the elements of $S_m$ can permute the second coordinates of the elements of $\mathcal B$. 
This implies that the action of $S_m$ on $\mathcal A$ has 
$\frac{(Nm)!}{(N!)^m m!}$ orbits. 
In particular we note that this number is always an integer. 
By now taking into consideration the number of orbits of $\fl(\qo)$ under the action of $\G(\qo)$, we have the following lemma.

\begin{lemma}
\label{lemma:orbitspowers}
Let $\po = \qo\odot \qo \odot \dots \odot\qo=\qo^m$ for some prime polytope $\qo$ with respect to $\odot$. Let $N^\odot$ denote the rank of $\qo$,  and $k$  denote the number of orbits of $\G(\qo)$ on $\fl(\qo)$. Then the number of orbits of $\fl(\po)$ under the action of $\G(\po)$ is
$$k^m \frac{(Nm)!}{(N!)^m m!}.$$
\end{lemma}


We are now ready to compute the number of flag orbits of a product.

\begin{prop}
\label{orbitsposet}
Let $\po=  \qo_1^{m_1} \odot \qo_2^{m_2}\odot \dots \odot \qo_r^{m_r}$, where the $\qo_i$ are distinct prime posets with respect to $\odot$. 
Let $n_i^\odot$ be the rank of  $\qo_i$ and let $k_i$  denote the number of orbits of $\G(\qo_i)$ on $\fl(\qo_i)$.  Then the number of orbits of $\fl(\po)$ under the action of $\G(\po)$ is
$$
\Pi_{i=1}^r k_i^{m_i} \frac{(\sum_{i=1}^r m_in_i)!}{\Pi_{i=1}^r (n_i!)^{m_i} m_i!},
$$
\end{prop}

\begin{remark}
Note that in Lemmas~\ref{orbitstworelatprime} and \ref{lemma:orbitspowers}, Corollary~\ref{coro:orbitsrelatprimes} and Proposition~\ref{orbitsposet}, we never use the fact that the posets are strongly connected. Hence similar propositions for pre-polytopes (in the sense of \cite{arp}), and their products as posets hold.
\end{remark}

We would like to find out when a product of polytopes is a regular or a 2-orbit polytope. We start by analyzing the case when a product is regular. Then, the number of orbits of Proposition~\ref{orbitsposet} has to equal one.

Start by noticing that 
\begin{eqnarray}
\label{k}
k:= \Pi_{i=1}^r k_i^{m_i} \frac{(\sum_{i=1}^r m_in_i)!}{\Pi_{i=1}^r (n_i!)^{m_i} m_i!} 
=
\Pi_{i=1}^r k_i^{m_i} \frac{(\sum_{i=1}^r m_in_i)!}{\Pi_{i=1}^r (m_in_i)!}\Pi_{i=1}^r\frac{(m_in_i)!}{ (n_i!)^{m_i} m_i!}.
\end{eqnarray}
Since $\frac{(M_in_i)!}{(n_i!)^{m_i}m_i!}$ is an integer for every $i$, then $\frac{(\sum_{i=1}^r m_in_i)!}{\Pi_{i=1}^r (m_in_i)!}$ is also an integer.
Furthermore, whenever $r\geq 2$, then $\frac{(\sum_{i=1}^r m_in_i)!}{\Pi_{i=1}^r (m_in_i)!}>1$. Hence, if we want $k=1$, then $r=1$ (that is, $\po$ is a power of a prime poset $\qo$). Thus, $k$ becomes
$$
 k=k_1^{m_1} \frac{(m_1n_1)!}{ (n_1!)^{m_1} m_1!}.
$$
Again, since $\frac{(m_1n_1)!}{ (n_1!)^{m_1} m_1!}$ is an integer, this immediately implies that $k_1=1$ (that is, $\qo$ is regular) and that  $\frac{(m_1n_1)!}{ (n_1!)^{m_1} m_1!}=1$. 
The last equality holds if and only if either $m_1=1$ or $n_1=1$. 
In the first case, this implies that $\po$  is a prime poset. 
The second case implies that $\po = \qo^{m_1}$, where $\qo$ has maximal chains of size $2$. 

\begin{theorem}
Let $\po$ be a regular polytope. Then, $\po$ is prime with respect to all four products except in the following cases:
\begin{enumerate}
\item If $\po$ is an $n$-simplex, then $\po$ is not prime with respect to the join product. In fact, $\po = v \Join v \Join \dots, \Join v$, where $v$ is a $0$-polytope.
\item If $\po$ is an $n$-cube, then $\po$ is not prime with respect to the cartesian product. In fact, $\po = e \times e \times \dots \times e$, where $e$ is a $1$-polytope.
\item If $\po$ is an $n$-crosspolytope, then $\po$ is not prime with respect to the direct sum. In fact, $\po = e \oplus e \oplus \dots \oplus e$, where $e$ is a $1$-polytope.
\item  If $\po = \qo \ \square \ \qo \ \square \dots \square \ \qo$, where $\qo$ is a $2$-polytope, then $\po$ is a regular polytope that is not prime with respect to the topological product.
\end{enumerate}
\end{theorem}

\begin{proof}
Each of the cases follows from the above discussion and by the following facts. 
In every case $\po$ must be the product of identical copies of prime polytopes (with respect to the given product). 
The join product is a product of posets and hence, with the above notation, the size of a maximal chain of the poset $\qo$ coincides with the number of flags of $\qo$ as polytope, implying that the rank of $\po$ is zero. 
The cartesian product and the direct sum taking away one element are products of posets, hence, the size of a maximal chain of the poset $\qo$ is in fact one more than the rank of $\qo$, that is, the rank of $\po$ must be one. 
Finally, the topological product, when taking away the minimum and maximum elements, is a product of posets. Hence the rank of $\qo$ must in fact coincide with the number of elements in a maximal chain, when seen as a factor in the product. 
That is, $\qo$ has to be a $2$-polytope.
\end{proof}

Following a similar analysis we can obtain an analogous theorem for two-orbit polytopes.

\begin{theorem}
Let $\po$ be a two-orbit polytope. Then $\po$ is prime with respect to the four products, except in the case where $\po$ is a torus  $\{4,4\}_{(a,0),(0,b)}$. In this case $\po = \qo \ \square \ \ko$, where $\qo$ and $\ko$ are non-isomorphic $2$-polytopes, and $\po$ is prime with respect to the other four products.
\end{theorem}

\begin{proof}
If $\po$ is a two-orbit polytope, then $k$ in (\ref{k}) must equal 2. We divide the analysis into two cases, when $r=1$, and when $r>1$.

Suppose first that $\po = \qo^m$, with $\qo$ prime. Then, (setting $n_1=n$) $k = k_1^{m} \frac{(mn)!}{ (n!)^{m} m!}$. Since $\frac{(mn)!}{ (n!)^{m} m!} \in \mathbb{N}$, then either $k_1=1$ and $\frac{(mn)!}{ (n!)^{m} m!}=2$ or $k_1=2$ and $m=1$. The first case can never happen (as $\frac{(mn)!}{ (n!)^{m} m!}\neq 1$ implies $\frac{(mn)!}{ (n!)^{m} m!} >2$), and in the second case $\po$ is simply a prime two-orbit polytope.

Suppose now that $r>1$. 
Since $\frac{(\sum_{i=1}^r m_in_i)!}{\Pi_{i=1}^r (m_in_i)!} > 1$ and $k=2$, we have that $\frac{(\sum_{i=1}^r m_in_i)!}{\Pi_{i=1}^r (m_in_i)!} = 2$ and hence $\Pi_{i=1}^r k_i^{m_i}\Pi_{i=1}^r\frac{(m_in_i)!}{ (n_i!)^{m_i} m_i!} = 1$. As for every $i$, $\frac{(m_in_i)!}{ (n_i!)^{m_i} m_i!}$ is an integer, this in turn implies that every $k_i = 1$ and that for every $i$, $\frac{(m_in_i)!}{ (n_i!)^{m_i} m_i!} =1$.
Let $b_i:=n_im_i$, for every $i$. Then we have that $\frac{(b_1+b_2+ \dots b_r)!}{b_1!b_2!\dots b_r!} = 2$. 
The last equality holds if and only if $r=2$ and $b_1=b_2=1$. Hence, $m_1=m_2=n_1=n_2=1$.

As pointed out before, for the join product, the cartesian product and the direct product, $n_1=n_2=1$ implies that $\qo_1=\qo_2$ and they are either a $0$-polytope or a $1$-polytope. However, we are under the assumption that $\qo_1$ and $\qo_2$ are relatively primes.

Hence, we only have left the case when $\qo_1$ and $\qo_2$, as polytopes, have rank $2$ and they are relatively prime with respect to $\square$. Since all rank 2 polytopes are prime with respect to the topological product, then $\qo_1$ and $\qo_2$ are only required to be non isomorphic 2-polytopes. This establishes the theorem.
\end{proof}

\section{Products and monodromy groups}\label{sec:mono}

The monodromy groups of an abstract polytope encapsulates all the combinatorial information of the polytope (see \cite{hartely99},\cite{hubard2009monodromy}). However, monodromy groups of non-regular abstract polytopes have been proven hard to understand (see \cite{mixmono}). Here, we study some basic properties of the monodromy group of a product. 

To this end, we use the description of the flags of a product given in Section~\ref{sec:flags}. As we have seen throughout, the four products of polytopes studied in this paper behave very much alike. Here, we give the details of our proofs only for the join product. The details of the other three products can be recovered from this one by making small modifications.

Let $\po = \qo_1\Join \qo_2 \Join \dots \Join \qo_r$ be an $(n-1)$-polytope, where $n=n_1+\dots+n_r$ and $n_i-1$ is the rank of the polytope $\qo_i$. 
For convenience, throughout this section we shall make use of Lemma~\ref{flagbijection}, and write each $\Phi\in\fl(\po)$ as $(\Psi^{(1)}, \Psi^{(2)}, \dots, \Psi^{(r)}, a)$, where each $\Psi^{(j)} \in \fl(\qo_j)$ and $a\in {\mathcal A}$. (Recall that $\mathcal A$ is the set of ordered $n$-tuples $a=(a_0,a_1, \dots, a_{n-1})$ with $a_j\in \{1, \dots, r\}$ and such that each $j \in \{1, \dots, r\}$ appears exactly $n_j$ times in $a$.)
Hence, we regard the set of flag $\fl(\po)$ and the set $\fl(\qo_1) \times \fl(\qo_2) \times \dots \times \fl(\qo_r) \times {\mathcal A}$ as the same object and use one or the other indistinctly.

Note that we can regard $S_n$ as the permutation group on the symbols $\{0,1,\dots, n-1\}$, and hence $S_n$ acts on $\mathcal A$ in a natural way. That is, given $a=(a_0,a_1, \dots, a_{n-1})\in {\mathcal A}$ and $\alpha \in S_n$, then $a\alpha = (a_{0\alpha},a_{1\alpha}, \dots, a_{(n-1)\alpha})$.
Let $r_0^{(i)}, r_1^{(i)}, \dots , r_{n_i-2}^{(i)}$ be the generators of the monodromy group ${\mathcal M}(\qo_i):=M_i$ of $\qo_i$. 
Hence, each $r_{j}^{(i)}$ permutes every flag of $\qo_i$ with its $j$-adjecent one.
Let $M:=M_1\times M_2\times \dots\times M_r$. 

We shall start by showing that the wreath product $\W:=M \wr_{\mathcal A} S_n$ of $M$ by $S_n$ acting on $\mathcal A$ as described above, acts on the set $\fl(\po)$ in a faithful way.
Recall that if $w=(\{w_b\}_{b\in{\mathcal A}}, \alpha)$, $v=(\{v_b\}_{b\in{\mathcal A}}, \beta) \in \W$, then 
$$wv=\big(\{w_bv_{b\alpha}\}_{b\in{\mathcal A}}, \beta\alpha\big).$$

Let $\Phi = (\Psi^{(1)}, \Psi^{(2)}, \dots, \Psi^{(r)}, a) \in \fl(\po)$ and $w=(\{w_b\}_{b\in{\mathcal A}}, \alpha) \in \W$. 
Since for every $b\in{\mathcal A}$, we have that $w_b\in M$, then $w_b$ is an $r$-tuple $w_b=(w_b^{(1)}, w_b^{(2)}, \dots, w_b^{(r)})$, with $w_b^{(j)}\in M_j$.
Hence, the action on $w$ on $\Phi$ is given by
\begin{eqnarray}
\label{monoaction}
\Phi w = (\Psi^{(1)}w_a^{(1)}, \Psi^{(2)}w_a^{(2)}, \dots, \Psi^{(r)}w_a^{(v)}, a\alpha^{-1}).
\end{eqnarray}
It is not difficult to see that (\ref{monoaction}) in fact defines an action of $\W$ on $\fl(\po)$ and that if an element  $w=(\{w_b\}_{b\in{\mathcal A}}, \alpha)\in W$ fixes every flag of $\po$, then $\alpha=1_{S_n}$ and for each  $b\in {\mathcal A}$ we have that $w_b=(1_{M_1}, 1_{M_2}, \dots, 1_{M_r})$, implying that the action  is faithful.

Let $s_0, s_1, \dots s_{n-2}$ be the generators of the monodromy group of $\po$, $\Mon(\po)$. Since each $s_k$ permutes every flag of $\po$ with its $k$-adjacent one, in order to understand $\Mon(\po)$ we need to understand the flag adjacencies in $\po$.
Let $k\in\{0,1,\dots n-2\}$. 
Consider the $k$-adjacent flag to $\Phi$, $\Phi^k$. 
Then we can write $\Phi=\{\Phi_{-1}, \Phi_0, \dots \Phi_{n-3},\Phi_{n-2}\}$ and $\Phi^k=\{\Phi_{-1}, \dots \Phi_{k-1}, \Lambda,\Phi_{k+1},\dots, \Phi_{n-2}\}$, where each $\Phi_i$, $i=-1,\dots, n-1$ as well as $\Lambda$ are faces of the product. 
For each $i$, we write $\Phi_i=(F_i^1, F_i^2, \dots, F_i^r)$ and  $\Lambda=(G^1, G^2, \dots, G^r)$.
Using the definition of the order of the product $\po$, we observe that, for each $i$, $\Phi_i$ and $\Phi_{i+1}$ differ in exactly one element (in fact, they differ in their $a_i^{(\Phi)}$ element). 
Hence, $\Phi_{k-1}$ and $\Phi_{k+1}$ differ on at most two elements and on at least one. 
Our study then naturally splits into two cases: when $\Phi_{k-1}$ and $\Phi_{k+1}$ differ on one or two elements.


We start by assuming that $\Phi_{k-1}$ and $\Phi_{k+1}$ differ in exactly one element, say $F_{k-1}^j$. That is, $\Phi_{k-1}=(F_{k-1}^1, F_{k-1}^2, \dots, F_{k-1}^r)$ and $\Phi_{k+1}=(F_{k-1}^1, \dots, F_{k-1}^{j-1}, F_{k+1}^j,F_{k-1}^{j+1},\dots, F_{k-1}^r)$. 
This immediately implies that $F_{k-1}^i=F_k^i=G^i$ for all $i\neq j$, and that, in $\qo_j$, $F^j_{k-1} , F^j_k , F^j_{k+1}$ are three different faces that are incident and whose ranks are consecutive. The same holds true for $F^j_{k-1} , G_k , F^j_{k+1}$.
In other words, when 
$a_k=a_{k+1}=j$, then
\begin{eqnarray}
\label{eq:ad1}
\Phi^k = (\{\Psi^{(1)}, \dots, \Psi^{(a_k-1)},(\Psi^{(a_k)})^l,\Psi^{(a_k+1)}\dots \Psi^{(r)}\}, (a_0,a_1, \dots a_{n-1})),
\end{eqnarray} 
where $(\Psi^{(a_k)})^l$ is the $l$-adjacent flag to $\Psi^{(a_k)}$, and $l$ is the number of times that $a_i$ appears in the sequence $a_0, a_1, \dots a_{k-1}$.

Suppose now that $\Phi_{k-1}$ and $\Phi_{k+1}$ differ in exactly two elements, say on those corresponding to $j_0$ and $j_1$.
Then $F^i_{k-1}=F^i_k=F^i_{k+1}=G^i$ for all $i\neq j_0, j_1$.
Furthermore, either $$F^{j_0}_{k-1}=F^{j_0}_k\neq F^{j_0}_{k+1} \ \mathrm{and} \ F^{j_1}_{k-1}\neq F^{j_1}_k=F^{j_1}_{k+1},$$
or
$$F^{j_0}_{k-1}\neq F^{j_0}_k= F^{j_0}_{k+1} \ \mathrm{and} \ F^{j_1}_{k-1}= F^{j_1}_k\neq F^{j_1}_{k+1},$$
In other words, when $a_k\neq a_{k+1}$,
\begin{eqnarray}
\label{eq:ad2}
\Phi^k= (\{\Psi^{(1)}, \dots,  \Psi^{(r)}\}, (a_0,\dots a_{k-1}, a_{k+1}, a_k, a_{k+2}, \dots a_{n-1})).
\end{eqnarray}

We are now ready to relate the monodromy group of $\po$ with the wreath product $\W$. 
Given $\odot\in\{\Join,\times,\oplus,\square\}$ and $\po=\qo_1\odot\qo_2\odot\dots\odot\qo_r$ a product polytope,
if $n_i$ is the rank of the polytope $\qo_i$, then we define $n$ as,
\begin{eqnarray}
\label{n}
n:=\left \{
\begin{array}{ll}
\Sigma_{i=1}^r n_i + r \ &\mathrm{if} \ \odot=\Join;  \\
\Sigma_{i=1}^r n_i \ &\mathrm{if} \ \odot=\times, \oplus; \\
\Sigma_{i=1}^r n_i - r \ &\mathrm{if} \ \odot=\square.
\end{array}\right.
\end{eqnarray}

\begin{prop}
\label{prop:Monext}
Let $\odot\in\{\Join,\times,\oplus,\square\}$.
Given polytopes $\qo_1,\qo_2,\dots,\qo_r$ and $\po=\qo_1\odot\qo_2\odot\dots\odot\qo_r$, the monodromy group of $\po$, $\Mon(\po)$, can be embedded as a subgroup of the wreath product $\mathcal{W}=M \wr_{\mathcal A} S_n$, where $M$ is the direct product of the monodromy groups of the $\qo_i$ and $n$ is as in (\ref{n}). 
Moreover the projection on the second factor $\pi|_{\Mon(\po)}: \Mon(\po)\to S_{n}$ is surjective. 
\end{prop}

\begin{proof} 
We give a proof for when $\odot=\Join$, the other three cases are similar.
Since we know that $\W$ acts faithfully on $\fl(\po$), to settle the first part of the proposition, it is enough to show that we can embed each of the generators of $\Mon(\po)$ in $\W$.

Let $k\in\{0,\dots n-2\}$ be fixed, and let $\Sigma:=\{a\in{\mathcal A}\mid a_k\neq a_{k+1}\}$ and for each $j\in\{1,\dots, r\}$ let $\Sigma_j:=\{a\in{\mathcal A}\mid a_k=a_{k+1}=j\}$. Consider $w_k=(\{(w_b^{(1)}, w_b^{(2)}, \dots,w_b^{(r)})\}_{b\in{\mathcal A}}, \alpha)\in\W$, where $\alpha=(k,k+1)\in S_n$ and 
\begin{eqnarray*}
w_b^{(j)}=\left \{
\begin{array}{ll}
1_{M_j} &\mathrm{if} \ b\in\Sigma\cup\bigcup_{i\neq j}\Sigma_i, \\
& \\
r^{(j)}_l&\mathrm{if} \ b\in\Sigma_j,
\end{array}\right.
\end{eqnarray*}
where $l$ is the number of times that $j$ appears in the sequence $b_0, b_1, \dots b_{k-1}$.
Using (\ref{eq:ad1}) and (\ref{eq:ad2}) is straightforward to see that for every $\Phi\in\fl$, $\Phi w_k=\Phi^k$. 
Hence each generator of $\Mon(\po)$ can be embedded into $\W$, implying that $\Mon(\po)$ can be embedded as a subgroup of the wreath product $\mathcal{W}$.

Since $\pi(w_k)=(k,k+1)\in S_n$,  the second part of the proposition follows.
\end{proof}

\begin{coro}
Let $\odot\in\{\Join,\times,\oplus,\square\}$ and $\po=\qo_1\odot\qo_2\odot\dots\odot\qo_r$ be a polytope. Then the monodromy group of $\po$ is an extension of a symmetric group $S_n$, where $n$ is as in (\ref{n}).
\end{coro}

This corollary tells us that the monodromy group of a product is always an extension of a symmetric group. However, this extension does not always splits, and figuring out when it does is not always easy or straightforward. In what follows we show how can this be computed in some simple examples. For the remainder of this section, we let $\pi:\W\to S_n$ be the natural projection and $K$ be the kernel of the restriction of $\pi$ to $\Mon(\qo_1\odot\qo_2\odot\dots\odot\qo_r)$.

\subsection{On the monodromy group of pyramids}

Let $\po$ be an $n$-polytope and consider its pyramid $\Pyr(\po)=\po\Join v$. 
Let $S_{n+2}$ denote the symmetric group on the symbols $0,1, \dots, n+1$, and for $i=0,\dots,n+1$,  let $\sigma_i=(i, i+1)\in S_{n+2}$.
The set $\mathcal A$ is the set $\{e_0,e_1, \dots, e_{n+1}\}$, where $e_i$ is the vector with $n+2$ entires such that  the entire $(i+1)$ is $2$ and the rest of them are $1$.
Since $v$ is a $0$-polytope, its monodormy group is trivial. Hence $\Mon(\Pyr(\po))$ is embedded as a subgroup of the wreath product $\W=\Mon(\po) \wr_{\mathcal A} S_{n+1}$.

Let $r_0, \dots, r_{n-1}$ be the generators of $\Mon(\po)$. Following the proof of Proposition~\ref{prop:Monext}, we can see that the generators $s_0, \dots, s_{n}$ of $\Mon(\Pyr(\po))$ can be regarded as:
\begin{eqnarray}\label{genMONOpyr} s_i=(r_{i-1}, \dots , r_{i-1}, 1,1,r_i\dots,r_i, \sigma_i)\in\W,\end{eqnarray}
where the identity element $1$ of $\Mon(\po)$ is in the $(i+1)$ and $(i+2)$ entries.

Observe that for each $i\in\{0,\dots, n-2\}$,
$$ (s_is_{i+1})^3 = ( (r_{i-1}r_i)^3, \dots ,(r_{i-1}r_i)^3, 1,1,1, (r_{i}r_{i+1})^3, \dots, (r_{i}r_{i+1})^3, \epsilon),$$
where the identity element $1$ of $\Mon(\po)$ is in the $(i+1)$, $(i+2)$ and $(i+3)$ entries and $\epsilon$ denotes the identity of $S_{n+1}$. Hence, the order of $s_is_{i+1}$ is $lcm[3, p_{i-1}, p_i]$, the largest common multiple of $3$, $p_i$ and $p_{i-1}$, where $p_j$ is the order of $r_{j}r_{j+1}$ in $\Mon(\po)$.

Computing the kernel $K$ of the restriction of the projection $\pi:\W\to S_{n+2}$ to $\Mon(\po)$ is rather difficult in general. 
One can use similar techniques to the ones we shall use in Section~\ref{sec:MonoPrism} to show that when $\po$ is a $p$-gon (that is, the simplest case of the pyramid) then $K \cong (C_m)^4$, where $m=\frac{p}{gcd(3,p)}$. Hence, the monodromy group is an extension of $S_4$ by $(C_m)^4$. Moreover, the extension splits if and only if $p$ is not congruent to $0$ modulo $9$. 
We do not give the details of this here, as this group has been previously computed in \cite{berman2014monodromy}.

\subsection{On the monodromy group of prisms}\label{sec:MonoPrism}
Let $\po$ be an $n$-polytope and consider its prism $\Pri(\po)=\po\times e$. We start by making some general remarks about $\Pri(\po)$ to exemplify how the above discussion would apply to one example in the cartesian product, and then proceed to compute the monodromy group of the prism over a polygon as an extension of $S_3$.

Let $S_{n+1}$ denote the symmetric group on the symbols $1, \dots, n+1$, let $\epsilon$ denote the identity of $S_{n+1}$ and for $i=1,\dots,n$, let $\sigma_i=(i, i+1)\in S_{n+1}$.
The set $\mathcal A$ is the set $\{e_1,e_2, \dots, e_{n+1}\}$, where $e_i$ is the vector with $n+1$ entries such that the entry $i$ is $2$ and the rest of them are $1$.
Since $e$ is a $1$-polytope, its monodromy group is a cyclic group of order 2. We let $t$ denote its generator.
And let $r_0, \dots, r_{n-1}$ be the generators of $\Mon(\po)$. 

By Proposition~\ref{prop:Monext}, $\Mon(\Pri(\po))$ can be embedded in $\W=(\Mon(\po)\times C_2) \wr_{\mathcal A} S_{n+1}$. 
Note that in this case, the $i$-adjacency of the flags of $\Pri(\po)$ is not described anymore by (\ref{eq:ad1}) and (\ref{eq:ad2}). 
Let  $\Phi=(\Psi, \Lambda, a)$ be a flag of $\Pri(\po)$, where $\Psi \in \fl(\po)$, $\Lambda\in\fl(e)$ and $a\in {\mathcal A}$. Then, the $0$-adjacency of a flag $\Phi$ is determined only by the value of the first entry of $a$.
\begin{eqnarray*}
\Phi^0=\left \{
\begin{array}{ll}
(\Psi^0, \Lambda,a) &\mathrm{if} \ a\neq e_1;\\
(\Psi, \Lambda^0,a) &\mathrm{if} \ a=e_1 .
\end{array}\right.
\end{eqnarray*}
For $i>0$, the $i$-adjacency is similar as that in (\ref{eq:ad1}) and (\ref{eq:ad2}). There is a small modification that has to be done, obtaining that,
\begin{eqnarray*}
\Phi^i=\left \{
\begin{array}{ll}
(\Psi, \Lambda,e_{i+1}) &\mathrm{if} \ a=e_i ;\\
(\Psi, \Lambda,e_{i}) &\mathrm{if} \ a=e_{i+1}; \\
(\Psi^{i-1}, \Lambda,a) &\mathrm{if} \ a\in\{e_1,\dots, e_{i-1}  \};\\
(\Psi^{i}, \Lambda,a) &\mathrm{if} \ a\in\{e_{i+2}, \dots, e_{n+1} \}.
\end{array}\right.
\end{eqnarray*}

Hence, using this to modify the ideas of the proof of Proposition~\ref{prop:Monext}, if $s_0, \dots, s_n$ denote the generators of $\Mon(\Pri(\po))$, then 
\begin{eqnarray}\label{genMONOprism}
s_0 &=& \big( (1,t), (r_0,1), (r_0,1), \dots, (r_0,1), \epsilon \big), \\
s_i &=& \big((r_{i-1},1), \dots , (r_{i-1},1), (1,1),(1,1),(r_i,1)\dots,(r_i,1), \sigma_i\big), \ \mathrm{if} \ i>0, \nonumber
\end{eqnarray}
where the identity element $1$ of $\Mon(\po)$ is in the $i$ and $(i+1)$ entries.

Computing the kernel $K$ is not always easy and depends on the monodromy group of $\po$. 
In what follows, we compute $K$, whenever $\po$ is a $p$-gon, that is, $\Pri(\po)$ is  simply the prism over a polygon.
In \cite{hartley2012minimal} the monodromy groups of prisms over polygons were computed in terms of generators and relations. 
Here, we also have the generators and can infer the relations of the group, but we shall focus on computing such group as a split extension of $S_3$. 

Let  $\po$ be a $p$-gon, and let $\qo$ be the prism over $\po$, that is $\qo=\Pri(\po)$.  
By Proposition~\ref{prop:Monext}, $\Mon(\qo)$ can be embedded into the wreath product $\W=M\wr_{\mathcal A} S_3$, where $M=\Mon(\po) \times C_2$ and ${\mathcal A}=\{(2,1,1), (1,2,1), (1,1,2)\}$. Hence, $\Mon(\qo)$ is in fact an  extension of $S_3$ by the group $K$,  the kernel of the restriction of $\pi:\W\to S_3$ to $\Mon(\qo)$. 
Furthermore, the generators in (\ref{genMONOprism}) become:
\begin{eqnarray*}
s_0 &=& \big( (1,t), (r_0,1), (r_0,1),  \epsilon \big), \\
s_1 &=& \big((1,1),(1,1),(r_1,1), \sigma_1\big), \\
s_2 &=& \big((r_1,1),(1,1),(1,1), \sigma_2\big).
\end{eqnarray*}

We start by noticing that $s_1s_2$ has order 3. Moreover, observe that
$$s_0s_1=\big( (1,t), (r_0,1), (r_0r_1,1), \sigma_1\big),$$ and hence
$$(s_0s_1)^2 = \big( (r_0,t), (r_0,t), ((r_0r_1)^2,1), \epsilon\big).$$
From here is straightforward to see that the order of $s_0s_1$ is $4m$, where $m=\frac{p}{gcd(p,4)}$
Moreover, 
$$s_2(s_0s_1)^2s_2=\big( (r_1r_0r_1,t),  ((r_0r_1)^2,1), (r_0,t),\epsilon\big).$$
and so $(s_0s_1)^2, s_2(s_0s_1)^2s_2 \in K$. In fact, it is not too difficult to see that $s_0, (s_0s_1)^2$ and $ s_2(s_0s_1)^2s_2$  generate $K$. 

We now study the structure of the group $K$. In order to slightly simplify our notation, we let $a:=s_0$, $b:=(s_0s_1)^2$, $c:=s_2bs_2$ and $d:=s_1cs_1$. Hence, $K=\langle a,b,c\rangle$ and $d= abac\in K$.
Observe that 
\begin{eqnarray*}
b^2&=&\big( (1,1), (1,1), ((r_0r_1)^4,1), \epsilon\big); \\
c^2&=&\big( (1,1), ((r_0r_1)^4,1),(1,1), \epsilon\big); \\
d^2&=&\big( ((r_0r_1)^4,1), (1,1), (1,1), \epsilon\big). \\
\end{eqnarray*}
Hence, the group $H=\langle b^2, c^2, d^2\rangle < K$ is isomorphic to $(C_m)^3$ (this group $H$ actually coincides with the group $H$ in \cite[Section 6]{hartley2012minimal}).
It is straightforward to see (using the description of the generators as elements of $\W$) that $H$ is normal in $K$. 
Moreover, the elements of the quotient $K/H$ are simply $\{H, aH, bH, cH, abH, acH, bcH, abcH\}$, implying that $K/H\cong (C_2)^3$. In other words, $K$ is an extension of $(C_2)^3$ by $(C_m)^3$.

Whenever $m$ is odd (that is, if $p$ is different than $0$ mod $8$), then $a, b^m,c^m \notin K\setminus H$ generate the $(C_2)^3$, implying that the extension splits. Otherwise, there are no elements of $K\setminus H$ that generate the $(C_2)^3$ and the extension does not split. Hence, we have the following proposition.

\begin{prop}
Let $\qo$ be the prism over a $p$-gon. Then, $\Mon(\qo) \cong K \rtimes S_3$, where $K$ is an extension of $(C_2)^3$ by $(C_m)^3$, with $m=\frac{p}{gcd(p,4)}$. Furthermore, this extension splits whenever $p$ is not congruent to $0$ modulo $8$. In this case, 
$$\Mon(\qo) \cong \big( (C_2)^3 \rtimes (C_m)^3\big) \rtimes S_3.$$
\end{prop}

Once we know the structure of the monodromy group of prisms over polygons, we consider  prisms over some $3$-polytopes. Let $p$ be an integer, and $\po$ be a $3$-polytope such that all its vertex figures are isomorphic to $p$-gons (these polytopes are sometimes called  {\em uniform} maps, however, the notation is not standard as {\em uniform} in other contexts means that the polytope is vertex-transitive). Consider the prism over $\po$.  Then, the generators of $\Mon(\Pri(\po))$ are
\begin{eqnarray*}
&s_0=\big( (1,t), (r_0,1), (r_0, 1), (r_0,1), \epsilon \big), & s_1=\big( (1,1), (1,1), (r_1,1), (r_1,1), \sigma_1\big), \\
&s_2=\big( (r_1,1),(1,1), (1,1),  (r_2,1), \sigma_2\big),& s_3=\big( (r_2,1), (r_2,1),(1,1), (1,1),  \sigma_3\big) .
\end{eqnarray*}
Note that $\langle s_1, s_2, s_3 \rangle$ is isomorphic to the monodromy group of the pyramid over a $p$-gon. 
Although this suggests that computing the kernel $K$ and hence knowing the structure of the monodromy group of the prism is easy (as we have already done the work for the pyramid), this is far from true. The reason for this is that now $s_0 \in K$, so finding the generators of $K$ is not easy. It is true, however, that $K$ contains a normal subgroup isomorphic to $(C_{\frac{p}{gcd(3,p)}})^4$ and that the extension of $S_4$ by $K$ splits whenever $p$ is not $0$ modulo $9$ (since the elements we need to use to recover $S_4$ in $\Mon(\Pri(\po))$ are the same as the ones needed in the pyramid). In other words, we have the following proposition.

\begin{prop}
Let $\po$ be a $3$-polytope such that all its vertex-figures are isomorphic to a $p$-gon. Then the monodromy group of $\Pri(\po$), the prism over $\po$, is a split extension of $S_4$ by some normal group whenever $p$ is not congruent to $0$ modulo $9$.
\end{prop}

\subsection{On the monodromy group of topological products with a polygon}

Let $\po$ be an $n$-polytope, and consider $\square_\po:= \po \square \qo$, where $\qo$ is a $p$-gon. Note that $\square_\po$ has rank $n+1$.
The analysis of $\Mon(\square_\po)$ is very similar to that of $\Pri(\po)$.
The two main differences are that $\Mon(\square_\po)$ is now an extension of $S_n$ (as opposed to $S_{n+1}$), and that the generators $s_0$ and $s_n$ of $\Mon(\square_\po)$ are now
\begin{eqnarray*}
s_0= \big( (1,t_0), (r_0,1), (r_0,1), \dots, (r_0,1), \epsilon \big), \\
s_n= \big( (1,t_1), (r_{n-1},1), (r_{n-1},1), \dots, (r_{n-1},1), \epsilon \big),
\end{eqnarray*}
where $t_0, t_1$ are the generators of $\Mon(\qo)$ and $r_0, \dots r_{n-1}$ are the generators of $\Mon(\po)$.
%

Again, computing in general $K$ is not easy in general. 
However, whenever $\po$ is also a $2$-polytope, say a $q$-gon, is rather simple.
In this case, $n=2$, so $\Mon(\po\square\qo)$ is a split extension of $S_2$ by $K$.
Moreover,
\begin{eqnarray*}
s_0=\big( (1,t_0), (r_0,1), \epsilon \big), \ \ 
s_1=\big( (1,1), (1,1), \sigma_1 \big), \ \
s_2=\big( (1, t_1), (r_1,1), \epsilon \big).
\end{eqnarray*}
and
\begin{eqnarray*}
s_1s_0s_1=\big( (r_0,1), (1,t_0), \epsilon \big), \ \ 
s_1s_2s_1=\big( (r_1,1), (1,t_1), \epsilon \big).
\end{eqnarray*}
Hence, the kernel $K$ is generated by $s_0, s_2, s_1s_0s_1$ and $s_1s_2s_1$. 
A simple computation shows then that $\langle s_0, s_1s_2s_1\rangle \cong \langle s_2, s_1s_0s_1\rangle \cong D_{m}$, where $m=[p,q]$ is the least common multiple of $p$ and $q$.
Moreover, it is straightforward to see that these two groups commute implying that $K = (D_m)^2$ and 
$$\Mon(\po\square\qo)=(D_m)^2 \rtimes S_2.$$

In fact, it is not difficult to extend these techniques to show that, if $\qo_i$ is a $p_i$-gon, then
$$\Mon(\qo_1\square \qo_2 \square \dots \square \qo_r) \cong (D_p)^r \rtimes S_r,$$
where $p$ is the least common multiple of $p_1, \dots p_r$.

We note here that this result is not surprising at all, since the monodromy group of a polytope $\po$ is isomorphic to the minimal regular cover of $\po$, whenever such cover is unique (see for example \cite{hartley2012minimal}). It is easy to see that the minimal regular cover of $\qo_1\square \qo_2 \square \dots \square \qo_r$ is the regular polytope $\qo^r$, where $\qo$ is a $p$-gon ($p=lcm[p_1,\dots, p_r]$) and the power is taken over the $\square$-product.

\section*{Conlcuding remarks}

As we have pointed out before, computing the monodromy group of non-regular polytopes is a difficult task. In this paper we showed that by regarding some polytopes as products this task can be simplified. In particular we think that the computations needed to calculate the monodromy groups of prisms and pyramids over polygons are fairly easy, specially if one compares them to those of \cite{berman2014monodromy} and \cite{hartley2012minimal}. For this reason we strongly believe that the techniques used here can be extended in order to compute monodromy groups of other interesting products and think it is an interesting project to pursue. 


\section*{Acknowledgments}

The authors would like to thank Toma\v{z} Pisanski and Ricardo Strausz for suggesting studying certain products on polytopes and Ricardo Strausz and Deborah Oliveros for some discussions on the subject. 
The completion of this work was done while the second author was on sabbatical at the Laboratoire d'Informatique de l'\'Ecole Polytechnique. She thanks LIX and Vicent Pilaud for their hospitality and the program PASPA-DGAPA and to the UNAM, the support for this sabbatical stay. 
We also gratefully acknowledge financial support of the PAPIIT-DGAPA, under grant IN107015 and of CONACyT, under grant.

\end{document}